\documentclass[11pt,reqno]{amsart}
\usepackage{amssymb}
\usepackage{amsmath, amssymb, amsthm}
\usepackage{mathrsfs}

\newtheorem{theorem}{Theorem}[section]

\newtheorem{lemma}{Lemma}[section]
\newtheorem{remark}{Remark}[section]

\numberwithin{equation}{section}

\begin{document}
\title[Magnetohydrodynamic equations]{Blow-up criterion for the   compressible magnetohydrodynamic equations with vacuum}

\author{Shengguo Zhu}
\address[S. G. Zhu]{Department of Mathematics, Shanghai Jiao Tong University,
Shanghai 200240, P.R.China; School of Mathematics, Georgia Tech
Atlanta 30332, U.S.A.}
\email{\tt zhushengguo@sjtu.edu.cn}

\begin{abstract}
In this paper,  the  $3$D compressible   MHD equations with  initial vacuum or infinity electric conductivity  is considered. We  prove  that the $L^\infty$ norms of the deformation tensor $D(u)$ and the absolute temperature $\theta$ control the possible blow-up  (see \cite{y1}\cite{olga}\cite{zx}) of strong solutions, especially   for the non-resistive MHD system when the magnetic diffusion vanishes. This conclusion  means that if a solution of the compressible MHD equations is initially regular and loses its regularity at some later time, then the formation of singularity must be caused  by  losing  the bound of  $D(u)$ or  $\theta$ as the critical time approaches.  The viscosity coefficients are only restricted by the physical conditions. Our criterion  (see (\ref{eq:2.911})) is similar to \cite{pc} for $3$D incompressible Euler equations and  \cite{hup} for  $3$D compressible isentropic Navier-stokes equations.
\end{abstract}

\date{Feb 12, 2014}
\keywords{MHD,  infinity electric conductivity,   vacuum, blow-up criterion.\\
{\bf Acknowledgments. } Shengguo Zhu's research was supported in part
by Chinese National Natural Science Foundation under grant 11231006.
}

\maketitle

\section{Introduction}
Magnetohydrodynamics is that part of the mechanics of continuous media which studies the motion of electrically conducting media in the presence of a magnetic field. The dynamic motion of  fluid and  magnetic field interact strongly on each other, so the hydrodynamic and electrodynamic effects are coupled. The applications of magnetohydrodynamics cover a very wide range of physical objects, from liquid metals to cosmic plasmas, for example, the intensely heated and ionized fluids in an electromagnetic field in astrophysics  and plasma physics.
In $3$-D space, the compressible   magnetohydrodynamic equations in a domain $\Omega $ of $\mathbb{R}^3$   can be written as
\begin{equation}
\label{eq:1.2pp}
\begin{cases}
\displaystyle
H_t-\text{rot}(u\times H)=-\text{rot}\Big(\frac{1}{\sigma}\text{rot}H\Big),\\[6pt]
\displaystyle
\text{div}H=0,\\[6pt]
\displaystyle
\rho_t+\text{div}(\rho u)=0,\\[6pt]
\displaystyle
(\rho u)_t+\text{div}(\rho u\otimes u)
  +\nabla P=\text{div}\mathbb{T}+\text{rot}H\times H,\\[6pt]
\displaystyle
(\rho \theta)_t+\text{div}(\rho \theta u)-\kappa \triangle \theta+P\text{div}u
=\text{div} (u\mathbb{ T})-u\text{div} \mathbb{T}+\frac{1}{\sigma}|\text{rot}H|^2.
\end{cases}
\end{equation}

In this system,  $x\in \Omega$ is the spatial coordinate; $t\geq 0$ is the time;  $H=(H^{(1)},H^{(2)},H^{(3)})$ is the magnetic field;  $\text{rot}H=\nabla \times H$ denotes the rotation of the magnetic field;  $0< \sigma\leq \infty$ is the electric conductivity coefficient;  $\rho$ is the mass density;
$u=(u^{(1)},u^{(2)},u^{(3)})\in \mathbb{R}^3$ is the velocity of fluids;   $\kappa> 0$ is the thermal conductivity coefficient; $P$ is the pressure  satisfying 
\begin{equation}
\label{eq:1.3}
P=R \rho \theta, 
\end{equation}
where   $\theta$ is the absolute temperature,  $R$  is a  positive constant; $\mathbb{T}$ is the stress tensor:
\begin{equation}
\label{eq:1.4}
\mathbb{T}=2\mu D(u)+\lambda \text{div}u \mathbb{I}_3,\quad D(u)=\frac{\nabla u+(\nabla u)^\top}{2},
\end{equation}
where $D(u)$ is the deformation tensor, $\mathbb{I}_3$ is the $3\times 3$ unit matrix, $\mu$ is the shear viscosity coefficient, $\lambda$ is the bulk viscosity coefficient, $\mu$ and $\lambda$ are both real constants satisfying
\begin{equation}
\label{eq:1.5}
  \mu > 0, \quad \lambda+\frac{2}{3}\mu \geq 0,\end{equation}
which ensures the ellipticity of the   Lam$\acute{\text{e}}$  operator.
Although the electric field $E$ doesn't appear in  system (\ref{eq:1.2pp}),
it is indeed induced according to a relation 
$$
E=\frac{1}{\sigma}\text{rot}H- u\times H
$$
by moving the conductive flow in the magnetic field.

The aim of this paper is to give a blow-up criterion of strong solutions  to  system (\ref{eq:1.2pp}) in  a bounded,  smooth domain $\Omega \in \mathbb{R}^3$ with the initial condition:
\begin{align}
\label{fan1}(H,\rho, u,\theta)|_{t=0}=(H_0(x), \rho_0(x),  u_0(x), \theta_0(x)),\  x\in \Omega,
\end{align}
and the  Dirichlet,  Neumann boundary conditions for $(H,u,\theta)$:

\begin{align}
\label{fan2}   (H, u,\partial \theta /\partial n)|_{\partial \Omega}=(0,0,0), \quad \text{when }\ 0< \sigma < +\infty;\\
\label{fan3}
( u,\partial \theta /\partial n)|_{\partial {\Omega}}=(0,0),\quad \text{when }\ \sigma= +\infty,
\end{align}
￼where   $n$ is the unit outer normal vector  to $\partial \Omega$.
Actually, some similar  result for $\Omega=\mathbb{R}^3$ can be also obtained via the similar argument used in this paper.

Throughout this paper, we adopt the following simplified notations for the standard homogeneous and inhomogeneous Sobolev space:
\begin{equation*}\begin{split}
&D^{k,r}=\{f\in L^1_{loc}(\Omega): |f|_{D^{k,r}}=|\nabla^kf|_{L^r}<+\infty\},\\
&  D^k=D^{k,2},\   \|(f,g)\|_X=\|f\|_X+\|g\|_X,\ \|f\|_{1,0}=\|f\|_{H^1_0(\Omega)},\\
&   \|f\|_s=\|f\|_{H^s(\Omega)},\
 |f|_p=\|f\|_{L^p(\Omega)},\quad |f|_{D^k}=\|f\|_{D^k(\Omega)}.
\end{split}
\end{equation*}
 A detailed study of homogeneous Sobolev space may be found in \cite{gandi}.

As has been observed in \cite{jishan}, when vacuum appears,  in order to make sure that the IBVP (\ref{eq:1.2pp})-(\ref{fan1})  with (\ref{fan2}) or (\ref{fan3}) is well-posed, the lack of a positive lower bound of the initial mass density $\rho_0$ should be compensated with some initial layer compatibility condition on the initial data $(H_0,\rho_0, u_0,\theta_0)$, for strong solutions  \cite{jishan},  which can be shown as

\begin{theorem} \cite{jishan} \label{th5}
Let the constant $q\in (3,6]$, and the initial data  $(H_0, \rho_0,  u_0,\theta_0)$ satisfy
\begin{equation}\label{th78}
\begin{split}
&      \rho_0\geq 0,\  \rho_0\in W^{1,q},\ u_0\in H^1_0\cap H^2,\    \theta_0\in H^2,\ \text{div}H_0=0,
\end{split}
\end{equation}
and   the  following initial layer compatibility conditions:
\begin{equation}\label{th79}
\begin{cases}
\displaystyle
Lu_0+\nabla P_0- \text{rot} H_0\times H_0=\sqrt{\rho_0} g_1,\quad \text{for \ some }\ g_1 \in L^2,\\[6pt]
\displaystyle
-\kappa \triangle \theta_0-Q(u_0)-\frac{1}{\sigma}|\text{rot} H_0|^2=\sqrt{\rho_0} g_2,\quad \text{for \ some }\ g_2 \in L^2,
\end{cases}
\end{equation}
where
$$P_0=R\rho_0\theta_0,\  Lu_0=-\mu \triangle u_0-(\mu+\lambda)\nabla \text{div}u_0.$$
\begin{enumerate}
\item
If  $0<\sigma< +\infty$,  
$
H_0\in H^1_0\cap H^2
$,
 then there exists a small time $T_*$ and a unique solution $(H,\rho,u,\theta)$ to IBVP (\ref{eq:1.2pp})-(\ref{fan1}) with (\ref{fan2})  satisfying:
\begin{equation}\label{regs}\begin{split}
&\rho\in C([0,T_*];W^{1,q}),\ (H,u)\in C([0,T_*];H^1_0\cap H^2)\cap  L^2([0,T_*];D^{2,q}),\\
&\theta \in C([0,T_*]; H^2)\cap  L^2([0,T_*];D^{2,q}),\\
&(H_t, u_t,\theta_t)\in  L^2([0,T_*];D^1),\  (H_t,\sqrt{\rho}u_t, \sqrt{\rho}\theta_t)\in L^\infty([0,T_*];L^2).
\end{split}
\end{equation}
\item
If $\sigma= +\infty$, 
$  H_0\in W^{1,q}$,
 then there exists a small time $T_*$ and a unique solution $(H,\rho,u,\theta)$ to IBVP (\ref{eq:1.2pp})-(\ref{fan1}) with  (\ref{fan3})  satisfying
\begin{equation}\label{regsa}\begin{split}
&(H,\rho) \in C([0,T_*];W^{1,q}),\ u\in C([0,T_*];H^1_0\cap H^2)\cap  L^2([0,T_*];D^{2,q}),\\
&\theta \in C([0,T_*]; H^2)\cap  L^2([0,T_*];D^{2,q}),\\
& ( u_t,\theta_t)\in  L^2([0,T_*];D^1),\  (\sqrt{\rho}u_t, \sqrt{\rho}\theta_t)\in L^\infty([0,T_*];L^2).
\end{split}
\end{equation}
\end{enumerate}
\end{theorem}

Some analogous existence theorems of local strong solutions to the  compressible Navier-Stokes equations   have been previously established by Choe and Kim in  \cite{CK3}\cite{CK}\cite{guahu}.  In $3$-D space, Huang-Li-Xin obtained the well-posedness of global classical solutions with small energy but possibly large oscillations and vacuum to the  Cauchy problem for isentropic flow in  \cite{HX1}.  Some similar existence results also have been obtained for compressible MHD equations in \cite{jishan}\cite{mhd}.  However,  via the similar arguments used in \cite{y1}\cite{zx}\cite{xy},  it is reasonable to believe that the local strong solution to (\ref{eq:1.2pp})-(\ref{fan1}) with boundary condition (\ref{fan2}) or (\ref{fan3})   may cease to exist globally. 

 So, naturally,  we want to know  the mechanism of blow-up and the structure of possible singularities: what kinds of singularities will form in finite time and what is the main mechanism of possible breakdown of smooth solutions for the $3$-D compressible MHD equations with thermal conductivity? 
The similar question has been studied for the incompressible Euler equation  by Beale-Kato-Majda (BKM) in their pioneering work \cite{TBK}, which showed that the $L^\infty$-bound of vorticity $\nabla \times u$ must blow up. Later, Ponce  \cite{pc} rephrased the BKM-criterion in terms of the deformation tensor $D(u)$. However, the same result as \cite{pc} has been proved by Huang-Li-Xin \cite{hup} for  compressible isentropic Navier-Stokes equations, which can be shown: if  $0 < \overline{T} < +\infty$  is the
maximum time for strong solution, then
\begin{equation}\label{jia2}
\lim \sup_{ T \rightarrow \overline{T}} \int_0^T |D( u)|_{ L^\infty(\Omega)}\text{d}t=\infty,
\end{equation}
and for  the compressible non-isentropic  system, Fan-Jiang-Ou \cite{jishan1} proved that
\begin{equation}\label{jia3}
\lim \sup_{ T \rightarrow \overline{T}} \Big(\int_0^T |\nabla u|_{ L^\infty(\Omega)}\text{d}t+|\theta|_{L^\infty([0,T];L^\infty(\Omega))}\Big)=\infty
\end{equation}
under the   assumption $7\mu > \lambda$  on the viscosity coefficients.
Recently, the similar blow-up criterion has been obtained for the $3$-D non-resistive $(\sigma=+\infty)$ compressible isentropic MHD equations in Xu-Zhang \cite{gerui}:
\begin{equation}\label{jia3}
\lim \sup_{ T \rightarrow \overline{T}} \int_0^T |\nabla u|_{ L^\infty(\Omega)}\text{d}t=\infty.
\end{equation}

Therefore, it is an interesting question to ask whether $L^\infty$ norm of $D(u)$ still controls the possible blow-up  for  strong solutions to IBVP  (\ref{eq:1.2pp})--(\ref{fan3}) as in \cite{hup}\cite{pc} or not? However, under the assumption:
\begin{equation}\label{jia1}
0<\sigma <+\infty,\quad \mu > 4\lambda,
\end{equation}
  some result has  been proved by Lu-Du-Yao  \cite{duyi}, which can be shown: if  $0 < \overline{T} < +\infty$  is the
maximum time for strong solution, then
\begin{equation}\label{xiangdi33}
\lim \sup_{ T \rightarrow \overline{T}} \Big(\int_0^T |\nabla u|_{ L^\infty(\Omega)}\text{d}t+|\theta|_{L^\infty([0,T];L^\infty(\Omega))}\Big)=\infty,
\end{equation}
and the  assumption $\mu > 4\lambda$  has been removed by Chen-Liu \cite{mingtao}.

However,  $D(u)$  is exactly the symmetric part of $\nabla u$:
$$
\nabla u=D(u)+\frac{\nabla u-(\nabla u)^\top}{2}.
$$
So it is clear that the blow-up criterions shown in (\ref{jia3})-(\ref{xiangdi33}) for the compressible MHD equations is much stronger than the one in (\ref{jia2}). This is mainly due to the presences of magnetic momentum flux density tensor
$$
\frac{1}{2}|H|^2I_3-H\otimes H
$$
 in momentum equation $(\ref{eq:1.2pp})_4$,  and the magnetic energy flux density vector
$$
E\times H=\Big(\frac{1}{\sigma}\text{rot}H-u\times H\Big) \times H
$$
 in energy equation $(\ref{eq:1.2pp})_5$. To deal with  both these two nonlinear terms, we need to control the norms $(|H|_\infty, |\nabla H|_{2})$,  which are difficult to be  bounded by $|D(u)|_{L^1(0,T;L^\infty)}$ because of the strong coupling between $u$ and $H$ in magnetic equations $(\ref{eq:1.2pp})_1$,  and the lack of smooth mechanism of $H$ for the case $\sigma=+\infty$. These are unlike those for $(|\rho|_\infty, |\nabla \rho|_{2})$,  which  can be totally determined by $|\text{div}u|_{L^1(0,T;L^\infty)}$ due to the simple scalar hyperbolic structure of the continuity equation $(\ref{eq:1.2pp})_1$. So some new arguments need to be introduced to improve the results obtained above   for system  (\ref{eq:1.2pp}) with thermal conductivity.

However,  via a subtle estimate for the magnetic field $H$ and making full use of the mathematical structure of  the system (\ref{eq:1.2pp}),  our main results in the following two theorems have sucessfully removed  the stringent condition $0<\sigma < +\infty$ appeared in (\ref{jia1}), and  instead of (\ref{xiangdi33}), replacing the term $\nabla u $ with the deformation tensor $D(u)$.

\begin{theorem} \label{th8} Let $0<\sigma<+\infty$ and   $(H,\rho,u,\theta)$  be a strong solution  to IBVP (\ref{eq:1.2pp})- (\ref{fan1}) with (\ref{fan2})   obtained in Theorem \ref{th5}.
Then if $0< \overline{T} <\infty$ is the maximal time for the existence of $(H, \rho,u,\theta)$,  we have

\begin{equation}\label{eq:2.911}
\lim \sup_{ T \rightarrow \overline{T}} \Big (\int_0^T |D( u)|_{L^\infty(\Omega)}\text{d}t+|\theta|_{L^\infty([0,T];L^\infty(\Omega))}\Big)=\infty.
\end{equation}
\end{theorem}
\begin{remark}\label{rr1}
This  conclusion answers our question positively for compressible isentropic flow, that is, the  $L^\infty$ norm of $D(u)$ still controls the possible blow-up  for the corresponding strong solutions.
Moreover the assumption that $\Omega$ is  bounded  is not essential,  and our argument can be easily applied to the Cauchy problem (see \cite{mingtao}\cite{duyi}) via some slight modifications.  The same blow-up criterion as (\ref{eq:2.911}) is available. Some related result on Serrin-type blow-up criterion can be seen in Huang-Li \cite{huangxin}.
\end{remark}

And when the magnetic diffusion vanishes:
\begin{theorem} \label{th9} Let $\sigma=+\infty$  and  $(H,\rho,u,\theta)$  be a strong  solution  to IBVP (\ref{eq:1.2pp})- (\ref{fan1}) with (\ref{fan3})   obtained in Theorem \ref{th5}.
Then if  $0< \overline{T} <\infty$ is the maximal time for  the  existence of $(H, \rho,u,\theta)$, we have

\begin{equation}\label{eq:2.912}
\lim \sup_{ T \rightarrow \overline{T}} \Big (\int_0^T |D( u)|_{L^\infty(\Omega)}\text{d}t+|\theta|_{L^\infty([0,T];L^\infty(\Omega))}\Big)=\infty.
\end{equation}
\end{theorem}
\begin{remark}
If we only consider the compressible isentropic  flow, Theorem \ref{th9} has answered exactly the same question as above that whether we can replace  $\nabla u$ with the deformation tensor $D(u)$  when $\sigma=+\infty$ in the  blow-up criterion  (\ref{jia3}) or not, which is  firstly raised by Xu-Zhang in \cite{gerui}.
\end{remark}

The rest of  this paper is organized as follows.   In Section $2$,   we give the proof for  (\ref{eq:2.911}) when $0<\sigma <+\infty$, which improves the results obtained in \cite{mingtao}\cite{duyi} via replacing $\nabla u$ with the deformation tensor $D(u)$. In Section $3$,  we show that the same blow-up criterion also holds when magnetic diffusion vanishes, that is $\sigma=+\infty$, which removes  the stringent condition $0<\sigma < +\infty$. Finally, we give an appendix in Section $4$,  which will introduce a Poincar$\acute{\text{e}}$ type inequality
(see Lemma \ref{pang}) to deal with the absolute temperature $\theta$ under the Neumann boundary condition.

\section{Blow-up criterion (\ref{eq:2.911}) for $0<\sigma<+\infty$. }

 We first prove (\ref{eq:2.911}) for $0<\sigma<+\infty$. Let $(H, \rho, u,\theta)$ be the unique strong solution to    IBVP (\ref{eq:1.2pp})--(\ref{fan1}) with boundary condition (\ref{fan2}). We assume that the opposite holds, i.e.,
\begin{equation}\label{we11}
\begin{split}
\lim \sup_{T\mapsto \overline{T}} \big(|D( u)|_{L^1([0,T]; L^\infty(\Omega))}+|\theta|_{L^\infty([0,T];L^\infty(\Omega))}\big)=C_0<\infty.
\end{split}
\end{equation}
Firstly,  based on  $\text{div}H=0$,   there are some formulas for $(H,u)$:
\begin{equation}\label{zhoumou}
\begin{cases}
\displaystyle
\text{rot}(u\times H)=(H\cdot \nabla)u-(u\cdot \nabla)H-H\text{div}u,\ \text{rot}(\text{rot} H)=-\triangle H,\\[8pt]
\displaystyle
\text{rot}H\times H
=\text{div}\Big(H\otimes H-\frac{1}{2}|H|^2I_3\Big)=-\frac{1}{2}\nabla |H|^2+H\cdot \nabla H,\\[8pt]
\displaystyle
Q(u)=\text{div} (u\mathbb{ T})-u\text{div} \mathbb{T}=\frac{\mu}{2}|\nabla u+(\nabla u)^{\top}|^2+\lambda (\text{div}u)^2.
\end{cases}
\end{equation}

In the following, we will use the convention that $C$ denotes a generic finite positive constant only depending on $\mu$, $\lambda$, $\kappa$, $R$, $\Omega$, $|(g_1,g_2)|_2$ and $ \overline{T}$,  and is independent of $\sigma$. We write $C(\alpha)$ to emphasize that $C(\alpha)$ depends on  $\alpha$ if it is really needed, especially for $C(\sigma)$.

Next we need to show some  estimates for  $(H, \rho, u,\theta)$.
By assumption (\ref{we11}), we first show that  both the magnetic field $H$ and the mass density $\rho$ are both uniformly bounded.
 \begin{lemma}\label{s1} For any constant $r\geq 2$, we have
\begin{equation*}
\begin{split}
|\rho(t)|_{\infty}+|H(t)|_{\infty}+\frac{1}{\sigma}\int_0^T \int_{\Omega} |H|^{r-2}|\nabla H|^2\text{d}x\text{d}t \leq C, \quad 0\leq t< T,
\end{split}
\end{equation*}
where the finite constant $C>0$ only depends on $C_0$ and $T$ $(any\  T\in (0,\overline{T}])$.
 \end{lemma}
 \begin{proof}
Firstly,  multiplying $(\ref{eq:1.2pp})_1$ by $r |H|^{r-2} H$ ($r \geq 2$) and integrating  over $\Omega$ by parts, then we have
\begin{equation}\label{zhumw2}
\begin{split}
& \frac{d}{dt}|H|^r_r+\frac{r(r-1)}{\sigma}\int_{\Omega}|H|^{r-2}|\nabla H|^2\text{d}x\\
=& r \int_{\Omega} \big(H\cdot \nabla u-u\cdot \nabla H-H \text{div}u\big) \cdot H|H|^{r-2} \text{d}x\\
=&r \int_{\Omega} \big(H\cdot D( u)-u\cdot \nabla H-H \text{div}u\big) \cdot H|H|^{r-2} \text{d}x.
\end{split}
\end{equation}
Via integrating by parts, the second term on the right-hand side of (\ref{zhumw2}) can be written as
\begin{equation}\label{zhumw1}
\begin{split}
-r \int_{\Omega} \big(u\cdot \nabla H\big)\cdot H|H|^{r-2} \text{d}x=\int_{\Omega} \text{div}u |H|^r \text{d}x
\end{split}
\end{equation}
which, together with (\ref{zhumw2}), immediately yields
\begin{equation}\label{zhumw3}
\begin{split}
& \frac{d}{dt}|H|^r_r+\frac{r(r-1)}{\sigma}  \int_{\Omega} |H|^{r-2}|\nabla H|^2\text{d}x \leq (2r+1) |D( u)|_{\infty} |H|^r_r.
\end{split}
\end{equation}
So, from $r\geq 2$ and (\ref{zhumw3}), we quickly have
\begin{equation}\label{mou10}
 \frac{d}{dt}|H|_r\leq \frac{ (2r+1)}{r} |D( u)|_{\infty} |H|_r,
\end{equation}
hence, it follows from (\ref{we11})  and (\ref{zhumw3})-(\ref{mou10}) that 
\begin{equation*}
\begin{split}
\sup_{0\leq t \leq T}|H(t)|_{r}+\frac{1}{\sigma} \int_0^T  \int_{\Omega} |H|^{r-2}|\nabla H|^2\text{d}x \text{d}t\leq C, \quad 0\leq t< T,
\end{split}
\end{equation*}
where $C>0$ is independent of $r$. Therefore, letting $r\rightarrow \infty$ in the above inequality leads to the desired estimate of $|H|_{\infty}$. In the same way, we also obtain the bound of $|\rho|_{\infty}$ which indeed depends only on $\|\text{div}u\|_{L^1([0,T];L^{\infty}(\Omega))}$.

 \end{proof}

\begin{remark}\label{lian1}
According to the proof for Lemma \ref{s1},  it is obvious that we can also obtain 
\begin{equation}\label{xue1}
\begin{split}
|\rho(t)|_{\infty}+|H(t)|_{\infty}\leq C, \quad 0\leq t< T
\end{split}
\end{equation}
 where  $C$ is only dependent of $C_0$ and $T$ $(any\  T\in (0,\overline{T}])$, and certainly  is also independent of $\sigma$. That is to say, (\ref{xue1})  also  holds for the case $\sigma=+\infty$ (see Lemma \ref{s1c}).\end{remark}

The next estimate  follows from the standard energy estimate:

  \begin{lemma}\label{s2}
\begin{equation*}
\begin{split}
|\sqrt{\rho}u(t)|^2_{ 2}+|\sqrt{\rho}\theta(t)|^2_{ 2}+\int_{0}^{T}\Big(\frac{1}{\sigma}|\nabla H(t)|^2_{2}+|\nabla u(t)|^2_{2}+|\nabla \theta(t)|^2_{2}\Big)\text{d}t\leq C,\quad 0\leq t< T,
\end{split}
\end{equation*}
where the finite constant $C>0$ only depends on $C_0$ and $T$ $(any\  T\in (0,\overline{T}])$.

 \end{lemma}
\begin{proof} Firstly,   multiplying  $(\ref{eq:1.2pp})_4$ by $u$,  $(\ref{eq:1.2pp})_3$ by $\frac{|u|^2}{2}$  and the  $(\ref{eq:1.2pp})_1$ by $H$, then summing them together and integrating the resulting equation over $\Omega$ by parts, we have

\begin{align}
\label{2}
\frac{1}{2}\frac{d}{dt}\int_{\Omega} \Big(\rho |u|^2+H^2\Big) \text{d}x+\int_{\Omega} \Big(\frac{1}{\sigma}|\nabla H|^2+\mu |\nabla u|^2+(\lambda+\mu)(\text{div}u)^2\Big)\text{d}x
=\int_{\Omega}  P \text{div}u\text{d}x,
\end{align}
where we have used the fact:
\begin{equation}\label{zhu1}
\begin{split}
\int_{\Omega} \text{rot}H \times H \cdot u \text{d}x=\int_{\Omega} -\text{rot}(u \times H)  \cdot H        \text{d}x.
\end{split}
\end{equation}
Then according to Holder's inequality  and Young's inequality, we have
\begin{equation}\label{zhu2s}
\begin{split}
\int_{\Omega} P \text{div} u \text{d}x\leq C|P|_2|\nabla u|_2\leq \frac{\mu}{4}|\nabla u|^2_2+C,\\
\end{split}
\end{equation}
which, together with (\ref{2}), means that
\begin{equation}\label{zhu3s}
\begin{split}
&\frac{d}{dt}\int_{\Omega} \Big(\frac{1}{2}\rho |u|^2+\frac{1}{2}H^2\Big) \text{d}x+\int_{\Omega} \Big(\frac{1}{\sigma}|\nabla H|^2+\mu |\nabla u|^2+(\lambda+\mu)(\text{div}u)^2\Big)\text{d}x\leq C.
\end{split}
\end{equation}

Secondly, multiplying  $(\ref{eq:1.2pp})_5$ by $\theta$ and integrating over $\Omega$, we have
\begin{equation}
\label{liu2s}
\begin{split}&\frac{d}{dt}\int_{\Omega} \frac{1}{2}\rho |\theta|^2 \text{d}x+\kappa\int_{\Omega}  |\nabla \theta|^2 \text{d}x\\
\leq& C\int_{\Omega} \rho \theta^2 |\text{div}u|\text{d}x+C\int_{\Omega} |\nabla u|^2 |\theta|\text{d}x+\frac{1}{\sigma}\int_{\Omega} |\text{rot}H|^2 |\theta|\text{d}x\\
\leq & C\Big(1+|\nabla u|^2_2+\frac{1}{\sigma}|\nabla H|^2_2\Big).
\end{split}
\end{equation}
Then from (\ref{zhu3s})-(\ref{liu2s}),  Gronwall's inequality and Lemma \ref{s1},  we obtain the desired conclusions.
\end{proof}
\begin{remark}\label{lian2}
According to the proof for Lemma \ref{s2},  especially for (\ref{liu2s}),  it is obvious that we can also obtain 
\begin{equation}\label{xue2}
\begin{split}
|\sqrt{\rho}u(t)|^2_{ 2}+|\sqrt{\rho}\theta(t)|^2_{ 2}+\int_{0}^{T}\big(|\nabla u(t)|^2_{2}+|\nabla \theta(t)|^2_{2}\big)\text{d}t\leq C,\quad 0\leq t< T,
\end{split}
\end{equation}
where  $C$ is only dependent of $C_0$ and $T$ $(any\  T\in (0,\overline{T}])$, and certainly  is also independent of $\sigma$. That is to say, (\ref{xue2})  also  holds for the case $\sigma=+\infty$ (see Lemma \ref{s2cc}).\end{remark}

The next lemma will give a key estimate on $\nabla H$,  $\nabla \rho$ and $\nabla u$.
  \begin{lemma}\label{s4}
\begin{equation*}
\begin{split}
|\nabla u(t)|^2_{ 2}+|\nabla \rho(t)|^2_{ 2}+|\nabla H(t)|^2_2+\int_0^T\Big( | u|^2_{D^2}+\frac{1}{\sigma}|H|^2_{D^2}\Big)\text{d}t\leq C(\sigma),\quad 0\leq t<  T,
\end{split}
\end{equation*}
where the finite constant  $C(\sigma)>0$ only depends on $C_0$, $\sigma$ and $T$ $(any\  T\in (0,\overline{T}])$.
 \end{lemma}
\begin{proof}

Firstly, multiplying $(\ref{eq:1.2pp})_4$ by $\rho^{-1}\big(-Lu-\nabla P-\frac{1}{2}\nabla |H|^2+H\cdot \nabla H\big) $ and integrating the resulting equation over $\Omega$,  via (\ref{zhoumou}) we have
\begin{equation}\label{zhu6}
\begin{split}
&\frac{1}{2} \frac{d}{dt}\Big(\mu|\nabla u|^2_2+(\mu+\lambda)|\text{div}u|^2_2\Big)+\int_{\Omega}\rho^{-1}\big(-Lu-\nabla P-\frac{1}{2}\nabla |H|^2+H\cdot \nabla H\big)^2\text{d}x\\
=&-\mu\int_{\Omega} (u\cdot \nabla u) \cdot \nabla \times (\text{rot}u)\text{d}x+(2\mu+\lambda)\int_{\Omega} (u\cdot \nabla u) \cdot  \nabla \text{div}u\text{d}x\\
&-\int_{\Omega} (u\cdot \nabla u) \cdot \nabla P(\rho) \text{d}x-\int_{\Omega} (u \cdot \nabla u) \Big(\frac{1}{2} \nabla |H|^2-H\cdot \nabla H\Big)\text{d}x\\
&-\int_{\Omega} u_t \cdot \nabla P(\rho) \text{d}x-\int_{\Omega} u_t\cdot  \Big(\frac{1}{2} \nabla |H|^2-H\cdot \nabla H\Big)\text{d}x\equiv: \sum_{i=1}^{6} L_i,
\end{split}
\end{equation}
where we have used the fact that $\triangle u=\nabla\text{div}u-\nabla\times \text{rot}u$.

We now estimate each term in (\ref{zhu6}). Due to the fact that $\rho^{-1}\geq C^{-1} >0$,  from the standard $L^2$-theory of elliptic system, we find that
\begin{equation}\label{gaibian}
\begin{split}
&\int_{\Omega}\rho^{-1}\big|Lu+\nabla P+\nabla |H|^2-H  \cdot \nabla H\big|^2\text{d}x\\
\geq& C^{-1}|Lu|^2_2-C(|\nabla P|^2_2+|H|^2_{\infty}|\nabla H^2|_2)\\
\geq& C^{-1}|u|^2_{D^2}-C(|\nabla \rho|^2_2+|\nabla \theta|^2_2+|\nabla u|^2_2+|\nabla H|^2_2),
\end{split}
\end{equation}
where we have used  Lemma $2.2$ and $L$ is a strong elliptic operator. Next according to
\begin{equation*}
\begin{cases}
u\times \text{rot}u=\frac{1}{2}\nabla (| u|^2)-u \cdot \nabla u,\\[8pt]
\nabla\times(a\times b)=(b\cdot \nabla)a-(a \cdot \nabla)b+(\text{div}b)a-(\text{div}a)b,
\end{cases}
\end{equation*}
Holder's inequality, Gagliardo-Nirenberg inequality and Young's inequality,
 we  deduce
\begin{equation}\label{zhu10cvcv}
\begin{split}
|L_1|=&\mu\Big|\int_{\Omega} (u \cdot \nabla u) \cdot \nabla \times (\text{rot} u)\text{d}x\Big|
=\mu\Big| \int_{\Omega} \nabla \times (u\cdot \nabla u) \cdot \text{rot} u\text{d}x\Big|\\
=&\mu\Big| \int_{\Omega} \nabla \times (u\times \text{rot}u) \cdot \text{rot} u\text{d}x\Big|\\
=&\mu\Big| \frac{1}{2}\int_{\Omega} (\text{rot} u)^2\text{div}u\text{d}x-\int_{\Omega} \text{rot} u\cdot D(u)\cdot \text{rot} u \text{d}x\Big|
\leq C|D(u)|_\infty|\nabla u|^2_2,\\
|L_2|=&(2\mu+\lambda)\Big|\int_{\Omega}(u\cdot \nabla u) \cdot  \nabla \text{div}u\text{d}x\Big|\\
=&(2\mu+\lambda)\Big|-\int_{\Omega}\nabla u: (\nabla u)^\top \text{div}u\text{d}x+\frac{1}{2}\int_{\Omega} (\text{div}u)^3\text{d}x\Big|
\leq C|D(u)|_\infty|\nabla u|^2_2,\\
|L_3|=&\Big|\int_{\Omega} (u\cdot \nabla u) \cdot \nabla P \text{d}x\Big| \leq C| u|_6 |\nabla u|_{3}|\nabla P|_2\\
\leq& C(\epsilon)(|\nabla \theta|^2_2+|\nabla \rho|^2_2+1)|\nabla u|^2_2+\epsilon | u|^2_{D^2},\\
L_{4}=&-\int_{\Omega}  (u \cdot \nabla u) \Big(\frac{1}{2} \nabla |H|^2-H\cdot \nabla H\Big) \text{d}x
\leq C|\nabla H|_2|H|_{\infty}|\nabla u|_3|u|_6\\
\leq& C(\epsilon)|H|^2_{\infty}|\nabla H|^2_2|\nabla u|^2_2+\epsilon \|\nabla u\|^2_1
\leq C(\epsilon)(|\nabla H|^2_2+1)|\nabla u|^2_2+\epsilon | u|^2_{D^2},
\end{split}
\end{equation}
\begin{equation}\label{yuelai}
\begin{split}
L_{5}=&-\int_{\Omega} u_t \cdot \nabla P \text{d}x=\frac{d}{dt} \int_{\Omega}  P \text{div}u \text{d}x-\int_{\Omega}  P_t \text{div}u \text{d}x\\
=&\frac{d}{dt} \int_{\Omega}  P \text{div}u \text{d}x-R\int_{\Omega}\rho_t \theta \text{div}u \text{d}x-R\int_{\Omega}\rho \theta_t  \text{div}u\text{d}x\\
\leq &\frac{d}{dt} \int_{\Omega}  P \text{div}u \text{d}x+R\int_{\Omega}\nabla \rho \cdot u \theta \text{div}u \text{d}x+R\int_{\Omega}\rho \theta  (\text{div}u)^2\text{d}x\\
&+R \int_{\Omega}  \rho u\cdot \nabla \theta \text{div}u \text{d}x+R^2 \int_{\Omega}  \rho \theta (\text{div}u)^2 \text{d}x-\kappa R\int_{\Omega}  \triangle\theta  \text{div}u  \text{d}x \\
&-R \int_{\Omega}   Q(u)\text{div}u \text{d}x-\frac{R}{\sigma} \int_{\Omega} |\text{rot} H|^2\text{div}u \text{d}x\\
\leq & \frac{d}{dt} \int_{\Omega}  P \text{div}u \text{d}x+C|\theta|_{\infty}|\nabla \rho|_2 |\nabla u|_2 |\nabla u|_3+C(|D(u)|_\infty+|\rho \theta|_\infty)|\nabla u|^2_2\\
&+C|\rho|_\infty|\nabla \theta|_2 |\nabla u|_2 |\nabla u|_3+C|\nabla \theta|_2|\nabla \text{div}u|_2+C|D(u)|_\infty\Big(\frac{1}{\sigma}|\nabla H|^2_2\Big)\\
\leq &\frac{d}{dt} \int_{\Omega}  P \text{div}u \text{d}x+ C (1+|D(u)|_\infty)\Big(|\nabla u|^2_2+\frac{1}{\sigma}|\nabla H|^2_2\Big)\\
 &+ C(\epsilon) |\nabla u|^2_2(|\nabla \theta|^2_2+|\nabla \rho|^2_2)+\epsilon | u|^2_{D^2}+C,
\end{split}
\end{equation}
where $\epsilon> 0$ is a sufficiently small constant.
And for the last term on the right handside of (\ref{zhu6}), we have
\begin{equation}\label{yuelaia}
\begin{split}
L_{6}=&-\int_{\Omega}  u_t \cdot \Big(\frac{1}{2} \nabla |H|^2-H\cdot \nabla H\Big)\text{d}x\\
=& \frac{1}{2}\frac{d}{dt}\int_{\Omega} |H|^2  \text{div}u \text{d}x-\frac{d}{dt}\int_{\Omega} H \cdot \nabla u \cdot H \text{d}x\\
&-\int_{\Omega}  \text{div}u  H \cdot H_t \text{d}x+\int_{\Omega} H_t \cdot \nabla u \cdot H \text{d}x+\int_{\Omega} H\cdot \nabla u \cdot H_t \text{d}x,
\end{split}
\end{equation}
where  we have used the fact $\text{div}H=0$.  To deal with the last three terms on the right-hand side of $L_6$, we need to use 
$$
H_t=H \cdot \nabla u-u \cdot \nabla H-H\text{div}u+\frac{1}{\sigma}\triangle H.$$
Hence, similarly to the proof of the above estimate, we also have
\begin{equation}\label{zhu12vvvv}
\begin{split}
& -\int_{\Omega}    \text{div}u  H \cdot H_t \text{d}x\\
=&\int_{\Omega} - \text{div}u  H \cdot \Big(H \cdot \nabla u-u \cdot \nabla H-H\text{div}u+\frac{1}{\sigma}\triangle H\Big) \text{d}x\\
\leq & C|H|^2_\infty |\nabla u|^2_2+C|D(u)|_{\infty} |\nabla H|_2 |u|_6 |H|_3\\
&+C|D(u)|_{\infty}\Big(\frac{1}{\sigma}|\nabla H|^2_2\Big)+C|H|_\infty|u|_{D^2} \Big(\frac{1}{\sigma}|\nabla H|_2\Big)\\
\leq & C(\epsilon)(|D(u)|_{\infty}+1)\Big(|\nabla u|^2_2+\Big(1+\frac{1}{\sigma}\Big)|\nabla H|^2_2\Big)+\epsilon |u|^2_{D^2},
\end{split}
\end{equation}
\begin{equation}\label{qian1}
\begin{split}
&\int_{\Omega} H_t \cdot \nabla u \cdot H \text{d}x+\int_{\Omega} H\cdot \nabla u \cdot H_t \text{d}x\\
= &2\int_{\Omega} \Big(H \cdot \nabla u-u \cdot \nabla H-H\text{div}u+\frac{1}{\sigma}\triangle H\Big) \cdot \nabla u \cdot H\text{d}x\\
\leq & C|H|^2_\infty |\nabla u|^2_2+C|u|_{\infty} |\nabla u|_2 |\nabla H|_2 |H|_\infty\\
&+C|u|_{D^2}|H|_\infty \Big(\frac{1}{\sigma}|\nabla H|_2\Big)+C|D(u)|_\infty\Big(\frac{1}{\sigma}|\nabla H|^2_2\Big)\\
\leq & C(\epsilon)\Big(\Big(1+\frac{1}{\sigma}\Big)|\nabla H|^2_{2}+1\Big)(|\nabla u|^2_2+|D(u)|_\infty+1)+\epsilon |u|^2_{D^2},
\end{split}
\end{equation}
where we have used the fact that 
\begin{equation}\label{bao1}
\begin{split}
&\int_{\Omega} \triangle H \cdot \nabla u \cdot H \text{d}x=\int_{\Omega} \sum_{k=1}^3 \sum_{i,j=1}^3\partial_{kk} H^i \partial_ju^i H^j \text{d}x\\
=&-\int_{\Omega} \sum_{k=1}^3 \sum_{i,j=1}^3\Big(\partial_{k} H^i \partial_ju^i \partial_kH^j +\partial_{k} H^i \partial_{jk}u^i H^j \Big)\text{d}x\\
=&-\int_{\Omega} \sum_{k=1}^3 \sum_{i,j=1}^3\Big(\partial_{k} H^i \frac{\partial_ju^i+\partial_iu^j}{2} \partial_kH^j +\partial_{k} H^i \partial_{jk}u^i H^j \Big)\text{d}x.
\end{split}
\end{equation}
Then combining (\ref{zhu6})-(\ref{qian1})  and choosing $\epsilon >0 $ suitably small,   we have
\begin{equation}\label{zhu6qss}
\begin{split}
&\frac{1}{2} \frac{d}{dt}\int_{\Omega}\Big(\mu|\nabla u|^2+(\mu+\lambda)|\text{div}u|^2-\Big(P+\frac{1}{2}|H|^2\Big) \text{div}u+ H\cdot \nabla u\cdot H\Big)\text{d}x+C|\nabla^2 u|^2_2\\
\leq &C\Big(|\nabla u|^2_2+\Big(1+\frac{1}{\sigma}\Big)|\nabla H|^2_2+|\nabla \rho|^2_2+1\Big)(|\nabla u|^2_2+|\nabla \theta|^2_2+|D(u)|_\infty+1).
\end{split}
\end{equation}

Secondly, applying $\nabla$ to  $(\ref{eq:1.2pp})_3$ and multiplying the resulting equations by $2\nabla \rho$,  we have
\begin{equation}\label{zhu20}
\begin{split}
&(|\nabla \rho|^2)_t+\text{div}(|\nabla \rho|^2u)+|\nabla \rho|^2\text{div}u\\
=&-2 (\nabla \rho)^\top \nabla u \nabla \rho-2 \rho \nabla \rho \cdot \nabla \text{div}u\\
=&-2 (\nabla \rho)^\top D(u) \nabla \rho-2 \rho \nabla \rho \cdot \nabla \text{div}u.
\end{split}
\end{equation}
Then integrating (\ref{zhu20}) over $\Omega$, we have
\begin{equation}\label{zhu21}
\begin{split}
\frac{d}{dt}|\nabla \rho|^2_2
\leq& C(\epsilon)(|D( u)|_\infty+1)|\nabla \rho|^2_2+\epsilon |\nabla^2 u|^2_2.
\end{split}
\end{equation}

Thirdly, applying $\nabla $  to $(\ref{eq:1.2pp})_1$, due to 
\begin{equation}\label{mou6}
\begin{split}
A=&\nabla (H\cdot \nabla u)=(\partial_j H \cdot \nabla u^i)_{(ij)}+(H\cdot \nabla \partial_j u^i)_{(ij)},\\
B=&\nabla (u \cdot \nabla H)=(\partial_j u\cdot \nabla H^i)_{(ij)}+ (u\cdot \nabla \partial_j H^i)_{(ij)},\\
C=&\nabla(H \text{div}u)=\nabla H \text{div}u+H \otimes \nabla \text{div}u,\\
D=&\nabla \triangle u: \nabla H=\sum_{i=1}^3\sum_{j=1}^3\partial_j \triangle H^{i}\partial_j H^i,
\end{split}
\end{equation}
then  multiplying the resulting equation $\nabla (\ref{eq:1.2pp})_1$ by $2\nabla H$, we have
\begin{equation}\label{zhu20q}
\begin{split}
&(|\nabla H|^2)_t-2A:\nabla H+2 B: \nabla H+2C : \nabla H=\frac{2}{\sigma}D.
\end{split}
\end{equation}
Then integrating (\ref{zhu20q}) over $\Omega$, due to
\begin{equation}\label{zhu20qq}
\begin{split}
&\int_{\Omega}  A: \nabla H \text{d}x\\
=&\int_{\Omega}  \sum_{j=1}^3\sum_{i=1}^3 \sum_{k=1}^3  \partial_j H^k \partial_k u^i \partial_j H^i \text{d}x+\int_{\Omega}  \sum_{j=1}^3\sum_{i=1}^3 \sum_{k=1}^3  H^k \partial_{kj} u^i \partial_j H^i \text{d}x\\
=&\int_{\Omega}  \sum_{j=1}^3\sum_{i,k=1}^3   \partial_j H^k \frac{(\partial_k u^i +\partial_i u^k)}{2}\partial_j H^i \text{d}x+\int_{\Omega}  \sum_{j=1}^3\sum_{i=1}^3 \sum_{k=1}^3  H^k \partial_{kj} u^i \partial_j H^i \text{d}x\\
\leq& C|D(u)|_\infty |\nabla H|^2_2+C|H|_\infty|\nabla H|_2 |u|_{D^2},\\
&\int_{\Omega}  B: \nabla H \text{d}x\\
=&\int_{\Omega}  \sum_{j=1}^3\sum_{i=1}^3 \sum_{k=1}^3  \partial_j u^k \partial_k H^i \partial_j H^i \text{d}x+\int_{\Omega}  \sum_{j=1}^3\sum_{i=1}^3 \sum_{k=1}^3  u^k  \partial_{kj} H^i \partial_j H^i \text{d}x\\
=&\int_{\Omega}  \sum_{i=1}^3\sum_{j,k=1}^3    \partial_k H^i \frac{(\partial_j u^k+\partial_k u^j)}{2} \partial_j H^i \text{d}x+\frac{1}{2}\int_{\Omega}  \sum_{j=1}^3\sum_{i=1}^3 \sum_{k=1}^3  u^k  \partial_{k} ( \partial_j H^i)^2 \text{d}x\\
\leq& C|D(u)|_\infty |\nabla H|^2_2,\\
&\int_{\Omega}  C: \nabla H \text{d}x
=\int_{\Omega} \big(\text{div}u|\nabla H|^2+H\otimes \nabla \text{div}u :\nabla H\big) \text{d}x\\
\leq& C|D(u)|_\infty |\nabla H|^2_2+C|H|_\infty|\nabla H|_2 |u|_{D^2},\\
&\int_{\Omega} D\text{d}x\\
=&\int_{\Omega} \sum_{k=1}^3\sum_{i=1}^3\sum_{j=1}^3 \partial_j \partial_{kk} H^{i}\partial_j H^i \text{d}x=-\int_{\Omega} \sum_{i=1}^3\sum_{k,j=1}^3 | \partial_{jk} H^{i}|^2 \text{d}x=-|H|^2_{D^2},
\end{split}
\end{equation}
 we quickly obtain the following estimate:
\begin{equation}\label{zhu21q}
\begin{split}
\frac{d}{dt}|\nabla H|^2_2+\frac{2}{\sigma}| H|^2_{D^2}
\leq& C(\epsilon)(|D( u)|_\infty+1)|\nabla H|^2_2+\epsilon |\nabla^2 u|^2_2.
\end{split}
\end{equation}
Now we denote 
$
\Gamma=\mu|\nabla u|^2+(\mu+\lambda)|\text{div}u|^2+|\nabla \rho(t)|^2_{ 2}+|\nabla H(t)|^2_{ 2}
$,
then adding (\ref{zhu21}) and (\ref{zhu21q}) to (\ref{zhu6qss}),  and choosing $\epsilon >0 $ suitably small,  we deduce that 
\begin{equation}\label{zhugai1}
\begin{split}
& \frac{d}{dt}\Gamma+\Big(| u|^2_{D^2}+\frac{1}{\sigma}| H|^2_{D^2}\Big)\\
\leq & C\frac{d}{dt}\int_{\Omega}\Big(\Big(P+\frac{1}{2}|H|^2\Big) \text{div}u- H\cdot \nabla u\cdot H\Big)\text{d}x \\
&+C\Big(\Gamma+\frac{1}{\sigma}|\nabla H|^2_2\Big)(|\nabla u|^2_2+|\nabla \theta|^2_2+|D(u)|_\infty+1).
\end{split}
\end{equation}

Then from Gronwall's inequality we immediately obtain
\begin{equation}\label{zhugai2}
\begin{split}
&|\nabla u(t)|^2_{ 2}+|\nabla \rho(t)|^2_{ 2}+|\nabla H(t)|^2_{ 2}+\int_0^t\Big( | u|^2_{D^2}+\frac{1}{\sigma}|H|^2_{D^2}\Big)\text{d}t\\
\leq & C\exp \Big(\Big(1+\frac{1}{\sigma}\Big) \int_0^t(|\nabla u|^2_2+|\nabla \theta|^2_2+|D(u)|_\infty+1) \text{d}s\Big) \leq C(\sigma).
\end{split}
\end{equation}

 \end{proof}

\begin{remark}\label{lian3}
According to    the proof for Lemma \ref{s2}, especially for (\ref{zhugai1})-(\ref{zhugai2}),  it is obvious that we can also obtain 
\begin{equation}\label{xue3}
\begin{split}
|\nabla u(t)|^2_{ 2}+|\nabla \rho(t)|^2_{ 2}+|\nabla H(t)|^2_2+\int_0^T| u|^2_{D^2}\text{d}t\leq C,\quad 0\leq t<  T,
\end{split}
\end{equation}
where  $C$ is only dependent of $C_0$ and $T$ $(any\  T\in (0,\overline{T}])$, and certainly  is also independent of $\sigma$. That is to say, (\ref{xue3})  also  holds for the case $\sigma=+\infty$ (see Lemma \ref{s4cc}).\end{remark}

Next, we proceed to improve the regularity of $H$, $\rho$, $u$ and $\theta$. To this end, we first give some estimate on the terms $\nabla^2 H$ and  $\nabla^2 u$ based on the above estimates.

 \begin{lemma}\label{lem:4-1}
\begin{equation*}
\begin{split}
&|(H,u)(t)|^2_{  D^2}+|\sqrt{\rho} u_t(t)|^2_{2} +|H_t(t)|^2_2+|\nabla \theta(t)|^2_2+|\rho_t(t)|^2_2\\
&+ \int_{0}^{T}\big( |u_t|^2_{D^1}+ |\sqrt{\rho}\theta_t|^2_{2}+ |H_t|^2_{D^1}+ |\theta|^2_{D^2}\big)\text{d}s\leq C(\sigma), \quad 0\leq t \leq T,
\end{split}
\end{equation*}
where the finite constant  $C(\sigma)>0$ only depends on $C_0$, $\sigma$ and $T$ $(any\  T\in (0,\overline{T}])$.

 \end{lemma}

 \begin{proof}

From system $ (\ref{eq:1.2pp})$, (\ref{zhoumou}) and the standard regularity estimate for elliptic equations,  we have
\begin{equation}\label{mousuncc}
\begin{split}
|H|_{D^2}\leq& C(\sigma)(|H_t|_2+|\text{rot}(u\times H)|_2+|\nabla H|_2)\\
\leq &C(\sigma) (|H_t|_2+|H|_\infty|\nabla u|_2+|\nabla u|_2|\nabla H|^{\frac{1}{2}}_2|\nabla H|^{\frac{1}{2}}_6)\\
|u|_{D^2}\leq& C(|\rho u_t|_2+|\rho u\cdot \nabla u |_2+|\nabla P|_2+|\text{rot}H\times H|_2+|\nabla u|_2)\\
\leq & C(|\rho|^{\frac{1}{2}}_\infty|\sqrt{\rho}u_t|_2+|\rho|_\infty| u|_6|\nabla u|^{\frac{1}{2}}_2|\nabla u|^{\frac{1}{2}}_6)\\
&+C(|\rho|_\infty|\nabla \theta|_2+|\theta|_\infty|\nabla \rho|_2+|H|_\infty|\nabla H|_2),\\
|\theta|_{D^2}\leq &C(|\rho \theta_t|_2+|\rho u\cdot \nabla \theta |_2+|P \text{div}u |_2+|Q(u)|_2+|\nabla \theta|_2)+C\frac{1}{\sigma}||\text{rot}H|^2|_2\\
\leq & C(|\rho|^{\frac{1}{2}}_\infty|\sqrt{\rho}\theta_t|_2+|\rho|_\infty| u|_6|\nabla \theta|^{\frac{1}{2}}_2|\nabla \theta|^{\frac{1}{2}}_6)\\
&+C|\nabla u|_6|\nabla u|_3+C|\rho|_\infty|\theta|_\infty|\nabla u|_2+C\frac{1}{\sigma}|\nabla H|_3|\nabla H|_6,\\
  |\rho_t|_2\leq& C(|\rho \text{div}u |_2+|u\cdot \nabla \rho|_2)\leq C(|\rho|_\infty| \text{div}u |_2+|u|_\infty|\nabla \rho|_2),
\end{split}
\end{equation}
which,  together with    Lemmas \ref{s1}-\ref{s4} and Yong's inequality,  immediately 
implies  that 
\begin{equation}\label{mousun}
\begin{split}
|H|_{D^2}\leq& C(\sigma)(|H_t|_2+1),\  |\rho_t|_2\leq C(\sigma)(1+|u|_\infty)\leq C(\sigma)(1+\|\nabla u\|_1),\\
|u|_{D^2}\leq& C(\sigma)(|\sqrt{\rho}u_t|_2+|\nabla \theta|_2+1),\\
|\theta|_{D^2}\leq& C(\sigma)(|\sqrt{\rho}\theta_t|_2+|\nabla \theta|_2+|\nabla u|_6|\nabla u|_3+|\nabla H|_6|\nabla H|_3+1).
\end{split}
\end{equation}

Next differentiating $ (\ref{eq:1.2pp})_4$ with respect to $t$, we have
\begin{equation}\label{gh78}
\begin{split}
\rho u_{tt}+Lu_t=-\rho_tu_t -\rho_t u\cdot\nabla u-\rho u_t\cdot\nabla u-\rho u\cdot\nabla u_t-\nabla P_t
 +(\text{rot}H\times H)_t.
\end{split}
\end{equation}
Multiplying (\ref{gh78}) by $u_t$ and integrating the resulting equation over $\Omega$, we have
\begin{equation}\label{zhen4}
\begin{split}
&\frac{1}{2}\frac{d}{dt}\int_{\Omega}\rho |u_t|^2 \text{d}x+\int_{\Omega}(\mu|\nabla u_t|^2+(\lambda+\mu)(\text{div}u_t)^2) \text{d}x\\
=&  -\int_{\Omega}( \rho u \cdot \nabla  |u_t|^2- \rho u \nabla ( u \cdot \nabla u \cdot u_t)- \rho u_t \cdot \nabla u \cdot u_t+ P_t \text{div} u_t)\text{d}x\\
& +\int_{\Omega}H  \cdot H_t \text{div} u_t\text{d}x-\int_{\Omega}\big(H \cdot \nabla u_t \cdot H_t+H_t \nabla u_t \cdot H\big)\text{d}x
\equiv:\sum_{i=7}^{12}L_i.
\end{split}
\end{equation}

According to Lemmas \ref{we11}-\ref{s4}, Holder's inequality, Gagliardo-Nirenberg inequality and Young's inequality,  we deduce that
\begin{equation}\label{zhou6}
\begin{split}
L_7=&-\int_{\Omega} \rho u \cdot \nabla  |u_t|^2\text{d}x\\
\leq & C|\rho|^{\frac{1}{2}}_{\infty}|u|_{\infty}|\sqrt{\rho} u_t|_{2}|\nabla u_t|_{2}\leq C\|\nabla u\|^2_1|\sqrt{\rho} u_t|^2_{2}+\frac{\mu}{10} |\nabla  u_t|^2_2, \\
L_8=& -\int_{\Omega} \rho u \nabla ( u \cdot \nabla u \cdot u_t)\text{d}x\\
 \leq& C|\rho|_\infty\int_{\Omega} \big(| u|  |\nabla u|^2 |u_t|+| u|^2  |\nabla^2 u| |u_t|+| u|^2  |\nabla u| |\nabla u_t| \big)\text{d}x\\
\leq& C |u_t|_6 ||\nabla u|^2|_{\frac{3}{2}} |u|_{6}+C ||u|^2|_{3} |\nabla^2 u|_{2} |u_t|_{6}+C||u|^2|_{3} |\nabla u|_{6}  | \nabla u_t|_{2}\\
\leq& C \big( |\nabla u|^2_{3} |\nabla u|_{2}+ |\nabla u|^2_{2} \|\nabla u\|_{1}  | \big)|\nabla u_t|_{2}\\
\leq & C(\sigma)\|\nabla u\|_{1}| \nabla u_t|_{2} \leq \frac{\mu}{10}|\nabla u_t|^2_{2}+C(\sigma)\|\nabla u\|^2_{1},
\end{split}
\end{equation}
where we have used the fact that 
\begin{equation}\label{mou5}
||u|^2|_{3} \leq C|u|^2_{6}  \leq C |\nabla u|^2_{2},\quad |\nabla u|^2_{3} \leq C|\nabla u|_{2} |\nabla u|_{6} \leq C |\nabla u|_{2} \|\nabla u\|_{1}.
\end{equation}

And similarly, we also have 

\begin{equation}\label{mou4}
\begin{split}
L_9=&-\int_{\Omega} \rho u_t \cdot \nabla u \cdot u_t\text{d}x  \leq C |D(u)|_\infty |\sqrt{\rho}u_t|^2_{2},\\
L_{10}=&\int_{\Omega} P_t \text{div} u_t\text{d}x =R\int_{\Omega}(\rho_t \theta+\rho \theta_t)  \text{div} u_t \text{d}x\\
\leq  & \frac{\mu}{10}|\nabla u_t|^2_{2}+C(|\rho_t|^2_2+|\sqrt{\rho}\theta_t|^2_2),\\
L_{11}+L_{12}=&
\int_{\Omega}H  \cdot H_t \text{div} u_t\text{d}x-\int_{\Omega}\big(H \cdot \nabla u_t \cdot H_t+H_t\cdot \nabla u_t \cdot H\big)\text{d}x\\
\leq & C|H|_\infty |H_t|_2|\nabla u_t|_2\leq  \frac{\mu}{10}|\nabla u_t|^2_{2}+C|H_t|^2_{2}.
\end{split}
\end{equation}
Then combining the above estimates (\ref{zhou6})-(\ref{mou4}),  from (\ref{zhen4}) and (\ref{mousun}) we have
\begin{equation}\label{zhen5gvsd}
\begin{split}
&\frac{1}{2}\frac{d}{dt}\int_{\Omega}\rho |u_t|^2 \text{d}x+\int_{\Omega}|\nabla u_t|^2\text{d}x\\
\leq & C(\sigma)(\|\nabla u\|^2_1+|D(u)|_\infty+1)(|\sqrt{\rho}u_t|^2_{2}+1)+C|H_t|^2_2+C|\sqrt{\rho}\theta_t|^2_2.
\end{split}
\end{equation}

Secondly, multiplying  $ (\ref{eq:1.2pp})_5$ by $\theta_t$ and integrating over $\Omega$, we have

\begin{equation}
\label{liu2ssccbb}
\begin{split}
&\frac{\kappa}{2}\frac{d}{dt}\int_{\Omega} |\nabla\theta|^2 \text{d}x+\int_{\Omega} \rho |\theta_t|^2 \text{d}x\\
=& -\int_{\Omega} \rho u\cdot \nabla \theta \theta_t\text{d}x-\int_{\Omega} P\text{div}u \theta_t\text{d}x\\
&+\int_{\Omega} Q(u) \theta_t\text{d}x+\frac{1}{\sigma}\int_{\Omega} |\text{rot}H|^2 |\theta_t|\text{d}x
= \sum_{i=13}^{16} L_i.
\end{split}
\end{equation}
According to Lemmas \ref{we11}-\ref{s4}, Holder's inequality, Gagliardo-Nirenberg inequality and Young's inequality,  we deduce that
\begin{equation}\label{zhou6ccss}
\begin{split}
L_{13}=&-\int_{\Omega}\rho u\cdot \nabla \theta \theta_t\text{d}x\\
\leq&  C|\rho|^{\frac{1}{2}}_{\infty}|u|_{\infty}|\sqrt{\rho} \theta_t|_{2}|\nabla \theta|_{2}
\leq \frac{1}{4} |\sqrt{\rho} \theta_t|_{2}+C\|\nabla u\|^2_1 |\nabla  \theta|^2_2, \\
L_{14}=& -\int_{\Omega} P\text{div}u \theta_t \text{d}x\\
 \leq& C|\rho|^{\frac{1}{2}}_{\infty}|\theta|_{\infty}|\sqrt{\rho} \theta_t|_{2}|\nabla u|_{2}
\leq \frac{1}{4} |\sqrt{\rho} \theta_t|_{2}+C |\nabla  u|^2_2,\\
L_{15}=&\int_{\Omega} Q(u) \theta_t\text{d}x=\frac{d}{dt}\int_{\Omega} Q(u) \theta\text{d}x-\int_{\Omega} Q(u)_t \theta\text{d}x\\
\leq &\frac{d}{dt}\int_{\Omega} Q(u) \theta\text{d}x+C|\nabla u|^2_2 +\frac{\mu}{10} |\nabla u_t|^2_2,\\
L_{16}=&\frac{1}{\sigma}\int_{\Omega} |\text{rot}H|^2\theta_t\text{d}x\\
=&\frac{1}{\sigma}\frac{d}{dt}\int_{\Omega} |\text{rot}H|^2 \theta\text{d}x-\frac{1}{\sigma}\int_{\Omega} |\text{rot}H|^2_t \theta\text{d}x\\
\leq &\frac{1}{\sigma}\frac{d}{dt}\int_{\Omega} |\text{rot}H|^2 \theta\text{d}x+C(\sigma)|\nabla H|^2_2 +\frac{1}{10\sigma} |\nabla H_t|^2_2,
\end{split}
\end{equation}
which,  tegether with (\ref{liu2ssccbb}), implies that 
\begin{equation}
\label{liu2vbvbcc}
\begin{split}\frac{\kappa}{2}\frac{d}{dt}&\int_{\Omega} |\nabla\theta|^2 \text{d}x+\int_{\Omega} \rho |\theta_t|^2 \text{d}x
\leq \frac{d}{dt}\int_{\Omega} \Big(\frac{1}{\sigma}|\text{rot}H|^2+Q(u)\Big) \theta\text{d}x\\
&\qquad \qquad \quad+C\|\nabla u\|^2_1|\nabla \theta_t|^2_{2}
+\frac{\mu}{10} |\nabla u_t|^2_2+\frac{1}{10\sigma} |\nabla H_t|^2_2+C(\sigma).
\end{split}
\end{equation}

Thirdly, differentiating  $ (\ref{eq:1.2pp})_1$ with respect to $t$,  multiplying  by $H_t$ and integrating over $\Omega$, we have
\begin{equation}\label{zhumw2cc}
\begin{split}
&\frac{1}{2} \frac{d}{dt}|H_t|^2_2+\frac{1}{\sigma}\int_{\Omega} |\nabla H_t|^2\text{d}x\\
=&  \int_{\Omega} \big(H_t\cdot \nabla u+H\cdot \nabla u_t-u_t\cdot \nabla H)\cdot H_t\text{d}x\\
&- \int_{\Omega}(u \cdot \nabla H_t+H_t \text{div}u+H \text{div}u_t\big) \cdot H_t \text{d}x\\
\leq & C|D(u)|_\infty|H_t|^2_2+C|H|_\infty|\nabla u_t|_2 |H_t|_2+C|u_t|_6|\nabla H|_2|H_t|_3\\
\leq &C(\sigma)(|D(u)|_\infty+1)|H_t|^2_2+\frac{\mu}{10}|\nabla u_t|^2_2+\frac{1}{10\sigma}|\nabla H_t|^2_2.
\end{split}
\end{equation}
Then combining (\ref{zhen5gvsd}), (\ref{liu2vbvbcc}) and (\ref{zhumw2cc}),  we have

\begin{equation}\label{gv88}
\begin{split}
&\frac{d}{dt}\int_{\Omega}\Big(\rho |u_t|^2+|H_t|^2+|\nabla \theta|^2\Big) \text{d}x+\int_{\Omega}\Big(|\nabla u_t|^2+\rho|u_t|^2+\frac{1}{\sigma}|\nabla H_t|^2\Big) \text{d}x\\
\leq &C(\sigma)(|\sqrt{\rho}u_t|^2_2+|H_t|^2_2+|\nabla \theta|^2_2)(|D(u)|_\infty+\|\nabla u\|^2_1+1)\\
&+\frac{d}{dt}\int_{\Omega} \Big(\frac{1}{\sigma}|\text{rot}H|^2+Q(u)\Big) \theta\text{d}x+C(\sigma)(1+\|\nabla u\|^2_1).
\end{split}
\end{equation}

From the momentum equations  $ (\ref{eq:1.2pp})_4$,  for any $\tau \in (0,T)$,  we easily have
\begin{equation}\label{li9}
\begin{split}
|\sqrt{\rho}u_t(\tau)|^2_2\leq C\int_{\Omega} \rho |u|^2|\nabla u|^2(\tau)\text{d}x+C\int_{\Omega} \frac{|\nabla P+Lu- \text{rot}H\times H|^2}{\rho}(\tau)\text{d}x,
\end{split}
\end{equation}
due to the initial layer compatibility condition (\ref{th79}), letting $\tau\rightarrow 0$ in (\ref{li9}), we have
\begin{equation}\label{nvk33}
\begin{split}
\lim \sup_{\tau\rightarrow 0}|\sqrt{\rho}u_t(\tau)|^2_2 \leq C\int_{\Omega} \rho_0 |u_0|^2|\nabla u_0|^2\text{d}x+C\int_{\Omega} |g_1|^2\text{d}x\leq C.
\end{split}
\end{equation}

Then integrating (\ref{gv88}) over $(0,T)$ with respect to $t$, via (\ref{nvk33}) and  Gronwall's inequality, we deduce that
\begin{equation*}
\begin{split}
\big(|\sqrt{\rho} u_t|^2_2+|H_t|^2_2+|\nabla \theta|^2_2\big)(t) +\int_0^T\Big(|\nabla u_t|^2_2+|\sqrt{\rho}\theta_t|^2_2+\frac{1}{\sigma}|\nabla H_t|^2_2\Big) \text{d}t\leq C(\sigma),\quad 0<t \leq T,
\end{split}
\end{equation*}
which, together with (\ref{mousun}), gives the desired conclusions.

\end{proof}

Via some   Poincar$\acute{\text{e}}$ type inequality (see (\ref{star} ) or Lemma \ref{pang}) coming from \cite{lions}, we have the following estimate for $ |\theta|_{D^2}$:
 \begin{lemma}\label{lem:4-1zx}
\begin{equation*}
\begin{split}
|\sqrt{\rho} \theta_t(t)|^2_{2} +|\theta(t)|^2_{D^2}+ \int_{0}^{T} |\theta_t|^2_{D^1}\text{d}s\leq C(\sigma), \quad 0\leq t \leq T,
\end{split}
\end{equation*}
where the finite constant  $C(\sigma)>0$ only depends on $C_0$, $\sigma$ and $T$ $(any\  T\in (0,\overline{T}])$. \end{lemma}
\begin{proof}
Firstly, from (\ref{mousun}) and Lemma \ref{lem:4-1}, we quickly have
\begin{equation}\label{yuyu}
|\theta|_{D^2}\leq C(\sigma)(1+|\sqrt{\rho}\theta_t|_2).
\end{equation}

Next differentiating $ (\ref{eq:1.2pp})_5$ with respect to $t$, we have
\begin{equation}\label{gh78v}
\begin{split}
\rho \theta_{tt}-\kappa \theta_t=&-\rho_t\theta_t -\rho_t u\cdot\nabla \theta-\rho u_t\cdot\nabla \theta-\rho u\cdot\nabla \theta_t\\
&+ P_t\text{div}u+P\text{div}u_t
 +Q(u)_t+\frac{1}{\sigma}|\text{rot}H|^2_t.
\end{split}
\end{equation}
Multiplying (\ref{gh78v}) by $\theta_t$ and integrating  over $\Omega$, we have
\begin{equation}\label{zhen4v}
\begin{split}
&\frac{1}{2}\frac{d}{dt}\int_{\Omega}\rho |\theta_t|^2 \text{d}x+\kappa\int_{\Omega}|\nabla \theta_t|^2\text{d}x\\
=&  R\int_{\Omega} \rho |\theta_t|^2 \text{div} u\text{d}x+R\int_{\Omega} \rho_t \theta \text{div} u \theta_t\text{d}x+R\int_{\Omega} \rho\theta \text{div} u_t \theta_t\text{d}x\\
&+\int_{\Omega} Q(u)_t \theta_t\text{d}x-\int_{\Omega} \rho_t  u \cdot \nabla \theta \theta_t\text{d}x-\int_{\Omega} \rho_t |\theta_t|^2 \text{d}x\\
&+\int_{\Omega} \rho u_t \cdot \nabla \theta \theta_t\text{d}x+\frac{1}{\sigma}\int_{\Omega} |\text{rot}H|^2_t \theta_t\text{d}x\equiv:\sum_{i=17}^{24}L_i.
\end{split}
\end{equation}
According to Lemmas \ref{we11}-\ref{lem:4-1}, Holder's inequality, Gagliardo-Nirenberg inequality and Young's inequality,  we deduce that
\begin{equation}\label{zhou6cc}
\begin{split}
L_{17}=&R\int_{\Omega} \rho |\theta_t|^2 \text{div} u\text{d}x\leq  C|D(u)|_{\infty}|\sqrt{\rho} \theta_t|^2_{2},\\
L_{18}=& R\int_{\Omega} \rho_t \theta \text{div} u \theta_t \text{d}x=-R\int_{\Omega} \rho\theta |\text{div} u|^2 \theta_t \text{d}x-R\int_{\Omega} \theta\nabla \rho \cdot u  \text{div} u \theta_t \text{d}x\\
 \leq& C|\rho|_{\infty}|\theta|_{\infty}|\nabla u|_{3}|\nabla u|_{2} |\theta_t|_6+C|u|_{\infty}|\theta|_{\infty}|\nabla \rho|_2|\nabla u|_{3}|\theta_t|_6\\
\leq & C(\sigma)  \big(|\sqrt{\rho}\theta_t|_2+|\nabla \theta_t|_2\big)
\leq \frac{\kappa}{8}|\nabla \theta_t|^2_2+C(\sigma)  |\sqrt{\rho} \theta_t|^2_{2}+C(\sigma) ,\\
L_{19}=&R\int_{\Omega} \rho\theta \text{div} u_t \theta_t\text{d}x\leq C|\theta|_\infty|\rho|_3|\nabla u_t|_2|\theta_t|_6\\
\leq &C|\nabla u_t|_2(|\nabla \theta_t|_2+|\sqrt{\rho}\theta_t|_2) \leq  \frac{\kappa}{8}|\nabla \theta_t|^2 +C(\sigma)  |\nabla u_t|^2+C(\sigma)  |\sqrt{\rho} \theta_t|^2_{2},\\
L_{20}=&\int_{\Omega}Q(u)_t \theta_t\text{d}x\leq C|\nabla u|_3|\nabla u_t|_2| \theta_t|_6\\
\leq &C(\sigma) |\nabla u_t|_2(|\nabla \theta_t|_2+|\sqrt{\rho}\theta_t|_2)\leq \frac{\kappa}{8}|\nabla \theta_t|^2 +C(\sigma)  |\nabla u_t|^2+C(\sigma)  |\sqrt{\rho} \theta_t|^2_{2},\\
L_{21}=& -\int_{\Omega} \rho_t  u\cdot \nabla \theta \theta_t \text{d}x
 \leq C|u|_{\infty}|\nabla \theta|_{3}|\rho _t|_{2}|\theta_t|_{6}\\
\leq & C(\sigma)|\nabla \theta|^{\frac{1}{2}}_{2}|\nabla \theta|^{\frac{1}{2}}_{6}(|\nabla \theta_t|_2+|\sqrt{\rho}\theta_t|_2)\\
\leq& \frac{\kappa}{8} |\nabla \theta_t|_{2}+C(\sigma) |\sqrt{\rho}\theta_t|^2_2+C(\sigma)\|\nabla \theta\|^2_1
\leq   \frac{\kappa}{8} |\nabla \theta_t|_{2}+C(\sigma) |\sqrt{\rho}\theta_t|^2_2+C(\sigma),\\
L_{22}=& -\int_{\Omega} \rho_t |\theta_t|^2 \text{d}x=-2\int_{\Omega} \rho u \cdot \nabla \theta_t \theta_t  \text{d}x\\
 \leq& C|\rho|^{\frac{1}{2}}_{\infty}|u|_{\infty}|\sqrt{\rho} \theta_t|_{2}|\nabla \theta_t|_{2}
\leq \frac{\kappa}{8} |\nabla \theta_t|_{2}+C(\sigma)  |\sqrt{\rho}\theta_t|^2_2,\\
\end{split}
\end{equation}
\begin{equation}
\label{zhou6ccmm}\begin{split}
L_{23}=& -\int_{\Omega} \rho u_t\cdot \nabla \theta \theta_t \text{d}x
 \leq C|\rho|^{\frac{1}{2}}_{\infty}|\nabla \theta|_{2}| \theta_t|_{6}|\sqrt{\rho} u_t|_{3}\\
\leq&C(\sigma) |\sqrt{\rho} u_t|^{\frac{1}{2}}_{2}|\sqrt{\rho} u_t|^{\frac{1}{2}}_{6} (|\nabla \theta_t|_2+|\sqrt{\rho}\theta_t|_2)\\
\leq& \frac{\kappa}{8}|\nabla \theta_t|^2 +C(\sigma)  |\nabla u_t|^2+ C(\sigma)|\sqrt{\rho} \theta_t|^2_{2},\\
L_{24}=&\frac{1}{\sigma}\int_{\Omega} |\text{rot}H|^2_t \theta_t\text{d}x\leq C|\nabla H|_3|\nabla H_t|_2| \theta_t|_6\\
\leq &C(\sigma) |\nabla H_t|_2(|\nabla \theta_t|_2+|\sqrt{\rho}\theta_t|_2)\leq \frac{\kappa}{8}|\nabla \theta_t|^2 +C(\sigma)  |\nabla H_t|^2+C(\sigma)  |\sqrt{\rho} \theta_t|_{2},
\end{split}
\end{equation}
where we have used the fact  (\ref{yuyu}) and the  following Poincar$\acute{\text{e}}$ type inequality (see \cite{lions}):
\begin{equation}
\label{star}\begin{split}
|\theta_t|_6\leq C(|\sqrt{\rho}\theta_t|_2+(1+|\rho|_2))|\nabla \theta_t|_2,
\end{split}
\end{equation}
and its proof can be seen in Lemma \ref{pang}.
Then according to (\ref{zhen4v})-(\ref{zhou6ccmm}), we deduce that 
\begin{equation}
\label{liuvbnm}
\begin{split}
&\frac{1}{2}\frac{d}{dt} \int_{\Omega}\rho |\theta_t|^2 \text{d}x+\kappa\int_{\Omega}|\nabla \theta_t|^2\text{d}x\\
\leq& C(\sigma)(|D(u)|_\infty+1)|\sqrt{\rho} \theta_t|^2_{2}+C(\sigma)(| u_t|^2_{D^1}+| H_t|^2_{D^1}+1).
\end{split}
\end{equation}
From the energy equations  $ (\ref{eq:1.2pp})_5$,  for any $\tau \in (0,T)$,  we easily have
\begin{equation}\label{li9cc}
\begin{split}
|\sqrt{\rho}\theta_t(\tau)|^2_2\leq C\int_{\Omega} \rho |u|^2|\nabla \theta|^2(\tau)\text{d}x+C\int_{\Omega} \frac{|\kappa\triangle \theta+Q(u)+\frac{1}{\sigma}|\text{rot}H|^2|^2}{\rho}(\tau)\text{d}x,
\end{split}
\end{equation}
due to the initial layer compatibility condition (\ref{th79}), letting $\tau\rightarrow 0$ in (\ref{li9}), we have
\begin{equation}\label{nvk33cc}
\begin{split}
\lim \sup_{\tau\rightarrow 0}|\sqrt{\rho}\theta_t(\tau)|^2_2 \leq C\int_{\Omega} \rho_0 |u_0|^2|\nabla \theta_0|^2\text{d}x+C\int_{\Omega} |g_2|^2\text{d}x\leq C.
\end{split}
\end{equation}
Then according to Gronwall's inequality,  (\ref{yuyu}) and (\ref{nvk33cc}), we immediately obtain the desired conclusions. 
\end{proof}

Finally, the following lemma gives bounds of $|\rho|_{D^{1,q}}$,    $| H|_{D^{2,q}}$,  $| u|_{D^{2,q}}$ and   $ |\theta|_{D^{2,q}}$.
 \begin{lemma}\label{s7}
 \begin{equation}\label{zhu54}
\begin{split}
\|(\rho(t)\|_{W^{1,q}}+|\rho_t(t)|_q+\int_0^T(|H|^2_{D^{2,q}}+|u|^2_{D^{2,q}}+|\theta|^2_{D^{2,q}})\text{d}t\leq C(\sigma),
\quad 0\leq t<  T,
\end{split}
\end{equation}
where the finite constant  $C(\sigma)>0$ only depends on $C_0$, $\sigma$ and $T$ $(any\  T\in (0,\overline{T}])$.\end{lemma}
\begin{proof}

 Firstly, from  $ (\ref{eq:1.2pp})_3$ and $ (\ref{eq:1.2pp})_4$, the standard regularity estimate for elliptic equations and Lemmas \ref{we11}-\ref{lem:4-1zx}, we have
 \begin{equation}\label{zhu55}
\begin{split}
|\nabla^2 u|_q \leq &C|\rho u_t+\rho u\cdot \nabla u
  +\nabla P-\text{rot}H\times H|_q+C|\nabla u|_q\\
\leq&  C(\sigma)(1+|\nabla u_t|_2+|\nabla \rho|_q).
\end{split}
\end{equation}

Next,  applying $\nabla$ to  $(\ref{eq:1.2pp})_3$, multiplying the resulting equations by $q|\nabla \rho|^{q-2} \nabla \rho$, we have
\begin{equation}\label{zhu20cccc}
\begin{split}
&(|\nabla \rho|^q)_t+\text{div}(|\nabla \rho|^qu)+(q-1)|\nabla \rho|^q\text{div}u\\
=&-q |\nabla \rho|^{q-2}(\nabla \rho)^\top D( u) (\nabla \rho)-q \rho|\nabla \rho|^{q-2} \nabla \rho \cdot \nabla \text{div}u.
\end{split}
\end{equation}
Then integrating (\ref{zhu20cccc}) over $\Omega$, we immediately obtain
\begin{equation}\label{zhu200}
\begin{split}
\frac{d}{dt}|\nabla \rho|_q
\leq& C|D( u)|_\infty|\nabla \rho|_q+C|\nabla^2 u|_q,
\end{split}
\end{equation}
which means that
\begin{equation}\label{zhu28ss}
\begin{split}
\frac{d}{dt}|\nabla \rho|_q
\leq& C(\sigma)(1+|D( u)|_\infty)|\nabla \rho|_q+C(\sigma)(1+|\nabla u_t|^2_2).
\end{split}
\end{equation}

Then from Gronwall's inequality, we immediately have
$$
|\nabla \rho|_q+|\leq C(\sigma)\exp \Big(\int_0^t(1+|D( u)|_\infty) \text{d}s\Big)\leq C(\sigma), \quad 0\leq t \leq T.
$$
So via (\ref{zhu55}) and Lemmas \ref{s1}-\ref{lem:4-1zx}, we easily have
\begin{equation}\label{zhu25ss}
\begin{split}
\int_0^t|u(s)|^2_{D^{2,q}}\text{d}s
\leq& C(\sigma)\int_0^t(1+ |\nabla u_t(s)|^2_2)\text{d}s\leq C(\sigma),\quad 0\leq t\leq T.
\end{split}
\end{equation}

Thirdly, we consider the term $|H|_{D^{2,q}}$ and $|\theta|_{D^{2,q}}$, from 
$(\ref{eq:1.2pp})_1$,
$(\ref{eq:1.2pp})_5$, the standard regularity estimate for elliptic equations and Lemmas \ref{we11}-\ref{lem:4-1zx}, we quickly have
\begin{equation}\label{jia5}
\begin{split}
|H|_{D^{2,q}}\leq & C(\sigma)(|H_t-\text{rot}(u\times H)|_q+|\nabla H|_q)
\leq C(\sigma)(|\nabla H_t|_2+1),\\
|\theta|_{D^{2,q}} \leq & C(\big|\rho \theta_t+\rho  u\cdot \nabla \theta+P\text{div}u|_q
+||\nabla u|^2|_q+|\nabla |_q)+C\frac{1}{\sigma}|\text{rot}H|^2|_q\\
\leq & C(\sigma)(1+|\theta_t|_{D^1}+|u|_{D^{2,q}}+|H|_{D^{2,q}}),
\end{split}
\end{equation}
which, together with Lemma \ref{lem:4-1} and  (\ref{zhu25ss}),  implies the desired conclusions.
 \end{proof}

And this will be enough to extend the strong solution  $(H,\rho,u,\theta)$ beyond $t\geq \overline{T}$.

In truth, in view of the estimates obtained in  Lemmas \ref{s1}-\ref{s7}, we quickly know that the functions $(H,\rho,u,\theta)|_{t=\overline{T}} =\lim_{t\rightarrow \overline{T}}(H,\rho,u,\theta)$ satisfy the conditions imposed on the initial data (\ref{th78})-(\ref{th79}) with $H_0 \in H^1_0\cap H^2$. Therefore, we can take $(H,\rho,u,\theta)|_{t=\overline{T}}$ as the initial data and apply the local existence Theorem \ref{th5} to extend our local strong solution beyond $t\geq \overline{T}$. This contradicts the assumption on $\overline{T}$.

\section{Blow-up criterion (\ref{eq:2.912}) for  $\sigma=+\infty$}\ \\

Based on the estimates obtained in Section $2$, now we  prove (\ref{eq:2.912}) for $\sigma=+\infty$. Let $(H, \rho, u,\theta)$ be the unique strong solution to   the IBVP (\ref{eq:1.2pp})--(\ref{fan1}) with (\ref{fan3}). We assume that the opposite holds, i.e.,
\begin{equation}\label{we11c}
\begin{split}
\lim \sup_{T\mapsto \overline{T}} \big(|D( u)|_{L^1([0,T]; L^\infty(\Omega))}+|\theta|_{L^\infty([0,T];L^\infty(\Omega))}\big)=C_0<\infty.
\end{split}
\end{equation}

Next we need to show some  estimates for our strong solution $(H, \rho, u,\theta)$.
By assumption (\ref{we11c}), the proof of Lemma \ref{s1} and Remark \ref{lian1}, we easily show that   the magnetic field $H$ and the mass density $\rho$ are both uniformly bounded.
 \begin{lemma}\label{s1c}
\begin{equation*}
\begin{split}
|\rho(t)|_{\infty}+|H(t)|_{\infty}\leq C, \quad 0\leq t< T,
\end{split}
\end{equation*}
where the finite constant  $C>0$ only depends on $C_0$ and $T$ $(any\  T\in (0,\overline{T}])$.
 \end{lemma}

The next estimate directly  follows from the estimate in Lemma \ref{s2} and Remark \ref{lian2}:

  \begin{lemma}\label{s2cc}
\begin{equation*}
\begin{split}
|\sqrt{\rho}u(t)|^2_{ 2}+|\sqrt{\rho }\theta(t)|^2_2+\int_{0}^{T}\big(|\nabla u(t)|^2_{2}+|\nabla \theta(t)|^2_{2})\text{d}t\leq C,\quad 0\leq t< T,
\end{split}
\end{equation*}
where the finite constant  $C>0$ only depends on $C_0$ and $T$ $(any\  T\in (0,\overline{T}])$.
 \end{lemma}

The next lemma directly  follows from the proof for Lemma \ref{s4} and Remark \ref{lian3}:
  \begin{lemma}\label{s4cc}
\begin{equation*}
\begin{split}
|\nabla u(t)|^2_{ 2}+|\nabla \rho(t)|^2_{ 2}+|\nabla H(t)|^2_2+\int_0^T | u|^2_{D^2}\text{d}t\leq C,\quad 0\leq t<  T,
\end{split}
\end{equation*}
where the finite constant  $C>0$ only depends on $C_0$ and $T$ $(any\  T\in (0,\overline{T}])$.
 \end{lemma}

Next, we proceed to improve the regularity of $H$, $\rho$, $u$ and $\theta$. To this end, we first drive some bounds on  $\nabla^2 u$ based on estimates above.

 \begin{lemma}\label{lem:4-1cc}
\begin{equation*}
\begin{split}
&|u(t)|^2_{  D^2}+|\sqrt{\rho} u_t(t)|^2_{2} +|\nabla \theta(t)|^2_2+|\rho_t(t)|^2_2\\
&+ \int_{0}^{T}\big( |u_t|^2_{D^1}+ |\sqrt{\rho}\theta_t|^2_{2}+ |\theta|^2_{D^2}\big)\text{d}s\leq C, \quad 0\leq t \leq T,
\end{split}
\end{equation*}
where the finite constant  $C>0$ only depends on $C_0$ and $T$ $(any\  T\in (0,\overline{T}])$.

 \end{lemma}
\begin{proof}
From system (\ref{eq:1.2pp}) with $\sigma=+\infty$, (\ref{we11c}), Lemmas \ref{s1c}-\ref{s4cc} and the  similar arguments used in the derivation of (\ref{mousuncc})-(\ref{mousun}), we have
\begin{equation}\label{moujia1}
\begin{split}
|H_t|_{2}\leq& C|\text{rot}(u\times H)|_2\\
\leq& C(|H|_\infty|\nabla u|_2+| u|_\infty|\nabla H|_2)\leq C(1+\|\nabla u\|_1),\\
|u|_{D^2}\leq& C(|\rho u_t|_2+|\rho u\cdot \nabla u |_2+|\nabla P|_2+|\text{rot}H\times H|_2+|\nabla u|_2)\\
\leq& C(|\sqrt{\rho}u_t|_2+|\nabla \theta|_2+1),\\
|\theta|_{D^2}\leq &C(|\rho \theta_t|_2+|\rho u\cdot \nabla \theta |_2+|P \text{div}u |_2+|Q(u)|_2+|\nabla \theta|_2)\\
\leq& C(|\sqrt{\rho}\theta_t|_2+|\nabla \theta|_2+|\nabla u|_6|\nabla u|_3+|\nabla H|_6|\nabla H|_3+1),\\
  |\rho_t|_2\leq& C(|\rho \text{div}u |_2+|u\cdot \nabla \rho|_2)\leq C(1+\|\nabla u\|_1).
\end{split}
\end{equation}

Next 
multiplying (\ref{gh78}) by $u_t$ and integrating the resulting equation over $\Omega$, we have
\begin{equation}\label{jiamou5}
\begin{split}
&\frac{1}{2}\frac{d}{dt}\int_{\Omega}\rho |u_t|^2 \text{d}x+\int_{\Omega}(\mu|\nabla u_t|^2+(\lambda+\mu)(\text{div}u_t)^2) \text{d}x\\
=&  -\int_{\Omega}( \rho u \cdot \nabla  |u_t|^2- \rho u \nabla ( u \cdot \nabla u \cdot u_t)- \rho u_t \cdot \nabla u \cdot u_t+ P_t \text{div} u_t)\text{d}x\\
& +\int_{\Omega}H  \cdot H_t \text{div} u_t\text{d}x-\int_{\Omega}\big(H \cdot \nabla u_t \cdot H_t+H_t \nabla u_t \cdot H\big)\text{d}x
\equiv:\sum_{i=7}^{12}L'_i.
\end{split}
\end{equation}
We have to point out that, compared with the relation (\ref{zhen4}), we have 
$$
L_i=L'_i,\quad \text{for} \quad i=7,...,12
$$
in the sense of  the form.
Then similarly to the derivarion of (\ref{zhou6})-(\ref{mou4}), via Lemmas \ref{s1c}-\ref{s4cc} and (\ref{moujia1}), we have
\begin{equation}\label{zhou6zz}
\begin{split}
L'_7\leq& C\|\nabla u\|^2_1|\sqrt{\rho} u_t|^2_{2}+\frac{\mu}{10} |\nabla  u_t|^2_2, \\
L'_8
\leq& C \big( |\nabla u|^2_{3} |\nabla u|_{2}+ |\nabla u|^2_{2} \|\nabla u\|_{1}  | \big)|\nabla u_t|_{2}\\
\leq&  C\|\nabla u\|_{1}| \nabla u_t|_{2} \leq \frac{\mu}{10}|\nabla u_t|^2_{2}+C\|\nabla u\|^2_{1},\\
L'_9 \leq& C |D(u)|_\infty |\sqrt{\rho}u_t|^2_{2},\\
L'_{10}
\leq  & \frac{\mu}{10}|\nabla u_t|^2_{2}+C(|\rho_t|^2_2+|\sqrt{\rho}\theta_t|^2_2),\\
L'_{11}+& L'_{12}
\leq  C|H|_\infty |H_t|_2|\nabla u_t|_2\\
\leq&  \frac{\mu}{10}|\nabla u_t|^2_{2}+C|H_t|^2_{2}\leq \frac{\mu}{10}|\nabla u_t|^2_{2}+ C(1+\|\nabla u\|^2_1).
\end{split}
\end{equation}
Then combining the above estimates (\ref{jiamou5})-(\ref{zhou6zz}),  we have
\begin{equation}\label{zhen5gvsdzz}
\begin{split}
&\frac{1}{2}\frac{d}{dt}\int_{\Omega}\rho |u_t|^2 \text{d}x+\int_{\Omega}|\nabla u_t|^2\text{d}x\\
\leq & C(\|\nabla u\|^2_1+|D(u)|_\infty+1)(|\sqrt{\rho}u_t|^2_{2}+1)+C|\sqrt{\rho}\theta_t|^2_2.
\end{split}
\end{equation}

Secondly, multiplying  $ (\ref{eq:1.2pp})_5$ with $\sigma=+\infty$ by $\theta_t$ and integrating over $\Omega$, we have

\begin{equation}
\label{liu2ssccbbzz}
\begin{split}
&\frac{\kappa}{2}\frac{d}{dt}\int_{\Omega} |\nabla\theta|^2 \text{d}x+\int_{\Omega} \rho |\theta_t|^2 \text{d}x\\
= &-\int_{\Omega} \rho u\cdot \nabla \theta \theta_t\text{d}x-\int_{\Omega} P\text{div}u \theta_t\text{d}x
+\int_{\Omega} Q(u) \theta_t\text{d}x
= \sum_{i=13}^{15} L'_i.
\end{split}
\end{equation}
We have to point out that, compared with the relation (\ref{liu2ssccbb}), we have 
$$
Li=L'_i,\quad \text{for} \quad i=13, 14,15
$$
in the sense of  the form.
Then similarly to the derivarion of (\ref{zhou6ccss}), via Lemmas \ref{s1c}-\ref{s4cc} and (\ref{moujia1}), we have
\begin{equation}\label{zhou6cczz}
\begin{split}
L'_{13}
\leq&  C|\rho|^{\frac{1}{2}}_{\infty}|u|_{\infty}|\sqrt{\rho} \theta_t|_{2}|\nabla \theta|_{2}
\leq \frac{1}{4} |\sqrt{\rho} \theta_t|_{2}+C\|\nabla u\|^2_1 |\nabla  \theta|^2_2, \\
L'_{14}
 \leq& C|\rho|^{\frac{1}{2}}_{\infty}|\theta|_{\infty}|\sqrt{\rho} \theta_t|_{2}|\nabla u|_{2}
\leq \frac{1}{4} |\sqrt{\rho} \theta_t|_{2}+C |\nabla  u|^2_2,\\
L'_{15}
\leq &\frac{d}{dt}\int_{\Omega} Q(u) \theta\text{d}x+C|\nabla u|^2_2 +\frac{\mu}{10} |\nabla u_t|^2_2,
\end{split}
\end{equation}
which,  tegether with (\ref{liu2ssccbbzz}), implies that 
\begin{equation}
\label{liu2vbvbcczz}
\begin{split}\frac{\kappa}{2}\frac{d}{dt}&\int_{\Omega} |\nabla\theta|^2 \text{d}x+\int_{\Omega} \rho |\theta_t|^2 \text{d}x\\
\leq & \frac{d}{dt}\int_{\Omega} Q(u)\theta\text{d}x+C\|\nabla u\|^2_1|\nabla \theta_t|^2_{2}
+\frac{\mu}{10} |\nabla u_t|^2_2+C.
\end{split}
\end{equation}

Then combining (\ref{zhen5gvsdzz}) and (\ref{liu2vbvbcczz}),   we have

\begin{equation}\label{gv88zz}
\begin{split}
&\frac{d}{dt}\int_{\Omega}\Big(\rho |u_t|^2+|\nabla \theta|^2\Big) \text{d}x+\int_{\Omega}|\nabla u_t|^2+\rho|u_t|^2\text{d}x\\
\leq &C(|\sqrt{\rho}u_t|^2_2+|\nabla \theta|^2_2)(|D(u)|_\infty+\|\nabla u\|^2_1+1)\\
&+\frac{d}{dt}\int_{\Omega}Q(u) \theta\text{d}x+C(1+\|\nabla u\|^2_1).
\end{split}
\end{equation}

And, completely same as the derivation of (\ref{nvk33}),  we have
\begin{equation}\label{nvk33ss}
\begin{split}
\lim \sup_{\tau\rightarrow 0}|\sqrt{\rho}u_t(\tau)|^2_2 \leq C\int_{\Omega} \rho_0 |u_0|^2|\nabla u_0|^2\text{d}x+C\int_{\Omega} |g_1|^2\text{d}x\leq C.
\end{split}
\end{equation}

Then integrating (\ref{gv88zz}) over $(0,T)$ with respect to $t$, via (\ref{nvk33ss}) and  Gronwall's inequality, we deduce that
\begin{equation*}
\begin{split}
\big(|\sqrt{\rho} u_t|^2_2+|\nabla \theta|^2_2\big)(t) +\int_0^T( |\nabla u_t|^2_2+|\sqrt{\rho}\theta_t|^2_2) \text{d}t\leq C,\quad 0<t \leq T,
\end{split}
\end{equation*}
which, together with (\ref{moujia1}), gives the desired conclusions.
\end{proof}
The next lemma is similar to   Lemma \ref{lem:4-1zx}:

 \begin{lemma}\label{lem:4-1cczx}
\begin{equation*}
\begin{split}
|\sqrt{\rho} \theta_t(t)|^2_{2} +|\theta(t)|^2_{D^2}+ \int_{0}^{T} |\theta_t|^2_{D^1}\text{d}s\leq C, \quad 0\leq t \leq T,
\end{split}
\end{equation*}
where the finite constant  $C>0$ only depends on $C_0$ and $T$ $(any\  T\in (0,\overline{T}])$.
 \end{lemma}
\begin{proof}
Firstly, from (\ref{moujia1}) and Lemmas \ref{s1c}-\ref{lem:4-1cc}, we quickly have
\begin{equation}\label{yuyuzz}
|\theta|_{D^2}\leq C(1+|\sqrt{\rho}\theta_t|_2).
\end{equation}

Next differentiating $ (\ref{eq:1.2pp})_5$  with respect to $t$ when $\sigma=+\infty$, we have
\begin{equation}\label{gh78vzz}
\begin{split}
\rho \theta_{tt}-\kappa \theta_t=&-\rho_t\theta_t -\rho_t u\cdot\nabla \theta-\rho u_t\cdot\nabla \theta-\rho u\cdot\nabla \theta_t\\
&+ P_t\text{div}u+P\text{div}u_t
 +Q(u)_t.
\end{split}
\end{equation}
Multiplying (\ref{gh78vzz}) by $\theta_t$ and integrating  over $\Omega$, we have
\begin{equation}\label{zhen4vzz}
\begin{split}
&\frac{1}{2}\frac{d}{dt}\int_{\Omega}\rho |\theta_t|^2 \text{d}x+\kappa\int_{\Omega}|\nabla \theta_t|^2\text{d}x\\
=&  R\int_{\Omega} \rho |\theta_t|^2 \text{div} u\text{d}x+R\int_{\Omega} \rho_t \theta \text{div} u \theta_t\text{d}x+R\int_{\Omega} \rho\theta \text{div} u_t \theta_t\text{d}x+\int_{\Omega} Q(u)_t \theta_t\text{d}x\\
&-\int_{\Omega} \rho_t |\theta_t|^2 \text{d}x
-\int_{\Omega} \rho_t  u \cdot \nabla \theta \theta_t\text{d}x+\int_{\Omega} \rho u_t \cdot \nabla \theta \theta_t\text{d}x\equiv:\sum_{i=17}^{23}L'_i.
\end{split}
\end{equation}
We have to point out that, compared with the relation (\ref{zhen4v}), we have 
$$
Li=L'_i,\quad \text{for} \quad i=17, ...,23
$$
in the sense of  the form.
Then similarly to the derivarion of (\ref{zhou6cc}), via Lemmas \ref{s1c}-\ref{lem:4-1cc} and (\ref{moujia1}), we have
\begin{equation}\label{zhou6cczzs}
\begin{split}
L_{17}\leq &  C|D(u)|_{\infty}|\sqrt{\rho} \theta_t|^2_{2},\\
L_{18}
\leq & C  \big(|\sqrt{\rho}\theta_t|_2+|\nabla \theta_t|_2\big)
\leq \frac{\kappa}{8}|\nabla \theta_t|^2_2+C(\sigma)  |\sqrt{\rho} \theta_t|^2_{2}+C ,\\
L_{19}
\leq &C|\nabla u_t|_2(|\nabla \theta_t|_2+|\sqrt{\rho}\theta_t|_2) \leq  \frac{\kappa}{8}|\nabla \theta_t|^2 +C |\nabla u_t|^2+C|\sqrt{\rho} \theta_t|^2_{2},\\
L_{20}
\leq &C |\nabla u_t|_2(|\nabla \theta_t|_2+|\sqrt{\rho}\theta_t|_2)\leq \frac{\kappa}{8}|\nabla \theta_t|^2 +C|\nabla u_t|^2+C |\sqrt{\rho} \theta_t|^2_{2},\\
L_{23}
 \leq & C|u|_{\infty}|\nabla \theta|_{3}|\rho _t|_{2}|\theta_t|_{6}
\leq  C|\nabla \theta|^{\frac{1}{2}}_{2}|\nabla \theta|^{\frac{1}{2}}_{6}(|\nabla \theta_t|_2+|\sqrt{\rho}\theta_t|_2)\\
\leq& \frac{\kappa}{8} |\nabla \theta_t|_{2}+C |\sqrt{\rho}\theta_t|^2_2+C\|\nabla \theta\|^2_1
\leq   \frac{\kappa}{8} |\nabla \theta_t|_{2}+C|\sqrt{\rho}\theta_t|^2_2+C,\\
L_{22}
 \leq& C|\rho|^{\frac{1}{2}}_{\infty}|u|_{\infty}|\sqrt{\rho} \theta_t|_{2}|\nabla \theta_t|_{2}
\leq \frac{\kappa}{8} |\nabla \theta_t|_{2}+C  |\sqrt{\rho}\theta_t|^2_2,\\
L_{23}
 \leq& C|\rho|^{\frac{1}{2}}_{\infty}|\nabla \theta|_{2}| \theta_t|_{6}|\sqrt{\rho} u_t|_{3}
\leq C|\sqrt{\rho} u_t|^{\frac{1}{2}}_{2}|\sqrt{\rho} u_t|^{\frac{1}{2}}_{6} (|\nabla \theta_t|_2+|\sqrt{\rho}\theta_t|_2)\\
\leq& \frac{\kappa}{8}|\nabla \theta_t|^2 +C  |\nabla u_t|^2+ C|\sqrt{\rho} \theta_t|^2_{2},
\end{split}
\end{equation}
where we have used the fact  (\ref{yuyuzz}) and the  following Poincar$\acute{\text{e}}$ type inequality (see \cite{lions}):
\begin{equation}
\label{starzz}\begin{split}
|\theta_t|_6\leq C|\sqrt{\rho}\theta_t|_2+C(1+|\rho|_2)|\nabla \theta_t|_2,
\end{split}
\end{equation}
and its proof can be seen in Lemma \ref{pang}.
Then according to (\ref{zhen4vzz})-(\ref{zhou6cczzs}), we deduce that 
\begin{equation}
\label{liuvbnm}
\begin{split}
&\frac{1}{2}\frac{d}{dt} \int_{\Omega}\rho |\theta_t|^2 \text{d}x+\kappa\int_{\Omega}|\nabla \theta_t|^2\text{d}x\\
\leq& C(|D(u)|_\infty+1)|\sqrt{\rho} \theta_t|^2_{2}+C(| u_t|^2_{D^1}+1).
\end{split}
\end{equation}
From the energy equations  $ (\ref{eq:1.2pp})_5$,  for any $\tau \in (0,T)$,  we easily have
\begin{equation}\label{li9cc}
\begin{split}
|\sqrt{\rho}\theta_t(\tau)|^2_2\leq C\int_{\Omega} \rho |u|^2|\nabla \theta|^2(\tau)\text{d}x+C\int_{\Omega} \frac{|\kappa\triangle \theta+Q(u)|^2}{\rho}(\tau)\text{d}x,
\end{split}
\end{equation}
due to the initial layer compatibility condition (\ref{th79}), letting $\tau\rightarrow 0$ in (\ref{li9cc}), we have
\begin{equation}\label{nvk33cczz}
\begin{split}
\lim \sup_{\tau\rightarrow 0}|\sqrt{\rho}\theta_t(\tau)|^2_2 \leq C\int_{\Omega} \rho_0 |u_0|^2|\nabla \theta_0|^2\text{d}x+C\int_{\Omega} |g_2|^2\text{d}x\leq C.
\end{split}
\end{equation}
Then according to Gronwall's inequality,  (\ref{yuyuzz}) and (\ref{nvk33cczz}), we immediately obtain the desired conclusions. 
\end{proof}

Finally, the following lemma gives bounds of $|\rho|_{D^1,q}$,    $|H|_{D^{1,q}}$,  $\nabla^2 u$ and   $\nabla^2 \theta$.
 \begin{lemma}\label{s7cc}
 \begin{equation*}
\begin{split}
\|(H,\rho)(t)\|_{W^{1,q}}+|(H_t,\rho_t)(t)|_{q}+\int_0^T(|u|^2_{D^{2,q}}+|\theta|^2_{D^{2,q}})\text{d}t\leq C,
\quad 0\leq t<  T,
\end{split}
\end{equation*}
where the finite constant  $C>0$ only depends on $C_0$ and $T$ $(any\  T\in (0,\overline{T}])$.
\end{lemma}
\begin{proof}
Firstly, from  $ (\ref{eq:1.2pp})_4$-$(\ref{eq:1.2pp})_5$, Lemmas \ref{s1c}-\ref{lem:4-1cczx} we  easily have
 \begin{equation}\label{zhu55cc}
\begin{split}
| u|_{D^{2,q}} \leq& C(|\rho u_t|_q+|\rho u\cdot \nabla u |_q+|\nabla P|_q+|\text{rot}H\times H|_q+|\nabla u|_q)\\
\leq &
 C(1+|u_t|_{D^1}+|\nabla \rho|_q+|\nabla H|_q),\\
|\theta|_{D^{2,q}}\leq &C(|\rho \theta_t|_q+|\rho u\cdot \nabla \theta |_q+|P \text{div}u |_q+|Q(u)|_q+|\nabla \theta|_q)\\
\leq& C(1+|\theta_t|_{D^1}+|u|_{D^{2,q}}).
\end{split}
\end{equation}

According to the proof in Lemma \ref{s7},  we immediately obtain
\begin{equation}\label{zhu200cc}
\begin{split}
\frac{d}{dt}|\nabla \rho|_q
\leq& C|D( u)|_\infty|\nabla \rho|_q+C|\nabla^2 u|_q,
\end{split}
\end{equation}
which means that
\begin{equation}\label{zhu28ssccbb}
\begin{split}
\frac{d}{dt}|\nabla \rho|_q
\leq& C(1+|D( u)|_\infty)(|\nabla \rho|_q+|\nabla H|_q)+C(1+|\nabla u_t|^2_2).
\end{split}
\end{equation}

Applying $\nabla$ to  $(\ref{eq:1.2pp})_1$, multiplying the resulting equation by $q\nabla H |\nabla H|^{q-2}$, we have
\begin{equation}\label{zhu20qs}
\begin{split}
&(|\nabla H|^2)_t-qA:\nabla H|\nabla H|^{q-2}+q B \nabla H|\nabla H|^{q-2}+qC : \nabla H|\nabla H|^{q-2}=0.
\end{split}
\end{equation}

Then integrating (\ref{zhu20qs}) over $\Omega$, due to
\begin{equation}\label{zhu20ccnm}
\begin{split}
&\int_{\Omega}  A: \nabla H |\nabla H|^{q-2}\text{d}x\\
=&\int_{\Omega}  \sum_{j=1}^3\Big(\sum_{i,k=1}^3  \partial_j H^k \partial_k u^i \partial_j H^i \Big)|\nabla H|^{q-2} \text{d}x+\int_{\Omega}  \sum_{j=1}^3\sum_{i=1}^3 \sum_{k=1}^3  H^k \partial_{kj} u^i \partial_j H^i |\nabla H|^{q-2}\text{d}x\\
=&\int_{\Omega}  \sum_{j=1}^3\Big(\sum_{i,k=1}^3  \partial_j H^k\frac{ (\partial_ku^i+\partial_iu^k)}{2}  \partial_j H^i \Big)|\nabla H|^{q-2} \text{d}x+\int_{\Omega}  \sum_{i,j,k=1}^3  H^k \partial_{kj} u^i \partial_j H^i |\nabla H|^{q-2}\text{d}x\\
\leq& C|D(u)|_\infty |\nabla H|^q_q+C|H|_\infty|\nabla H|^{q-1}_q |u|_{D^{2,q}},\\
&\int_{\Omega}  B: \nabla H|\nabla H|^{q-2} \text{d}x\\
=&\int_{\Omega}  \sum_{j=1}^3\sum_{i=1}^3 \sum_{k=1}^3  \partial_j u^k \partial_k H^i \partial_j H^i|\nabla H|^{q-2} \text{d}x+\int_{\Omega}  \sum_{j=1}^3\sum_{i=1}^3 \sum_{k=1}^3  u^k  \partial_{kj} H^i \partial_j H^i|\nabla H|^{q-2} \text{d}x\\
=&\int_{\Omega}  \sum_{i=1}^3\Big(\sum_{j,k=1}^3  \partial_j u^k \partial_k H^i \partial_j H^i \Big)|\nabla H|^{q-2} \text{d}x+\frac{1}{2}\int_{\Omega}  \sum_{k=1}^3  u^k\Big(\sum_{j,i=1}^3   \partial_{k} |\partial_j H^i|^2|\nabla H|^{q-2} \Big)\text{d}x\\
=&\int_{\Omega}  \sum_{i=1}^3\Big(\sum_{j,k=1}^3 \partial_k H^i    \partial_j u^k \partial_j H^i \Big)|\nabla H|^{q-2} \text{d}x+\frac{1}{2}\int_{\Omega}  \sum_{k=1}^3  u^k\Big(\sum_{j,i=1}^3 \partial_{k} |\nabla H|^2|\nabla H|^{q-2} \Big)\text{d}x\\
=&\int_{\Omega}  \sum_{i=1}^3\Big(\sum_{j,k=1}^3   \partial_k H^i \frac{(\partial_j u^k+\partial_ku^j)}{2} \partial_j H^i \Big)|\nabla H|^{q-2} \text{d}x+\frac{1}{q}\int_{\Omega}  \sum_{k=1}^3  u^k  \partial_{k} |\nabla H|^{q} \text{d}x\\
\leq& C|D(u)|_\infty |\nabla H|^q_q,
\end{split}
\end{equation}

\begin{equation}\label{zhu20vv}
\begin{split}
&\int_{\Omega}  C: \nabla H|\nabla H|^{q-2} \text{d}x
=\int_{\Omega} \big(\text{div}u|\nabla H|^q+(H\otimes \nabla \text{div}u) :\nabla H|\nabla H|^{q-2}\big) \text{d}x\\
\leq& C|D(u)|_\infty |\nabla H|^q_q+C|H|_\infty|\nabla H|^{q-1}_q |u|_{D^{2,q}},
\end{split}
\end{equation}
then we quickly obtain the following estimate
\begin{equation}\label{zhu21qs}
\begin{split}
\frac{d}{dt}|\nabla H|_q
\leq& C(|D( u)|_\infty+1)|\nabla H|_q+C | u|_{D^{2,q}},
\end{split}
\end{equation}
which means that
\begin{equation}\label{zhu28xx}
\begin{split}
\frac{d}{dt}|\nabla H|_q
\leq& C(1+|D( u)|_\infty)(|\nabla \rho|_q+|\nabla H|_q)+C(1+|\nabla u_t|^2_2).
\end{split}
\end{equation}

Then from (\ref{zhu28ssccbb}) and (\ref{zhu28xx}),  via Gronwall's inequality and (\ref{zhu55cc}) we quickly have
\begin{equation}\label{haoyue}
|\nabla H(t)|_q+|\nabla \rho(t)|_q+\int_0^T(|\nabla^2 u|^2_{q}+|\nabla^2 \theta|^2_{q}) \text{d}t\leq C,\quad 0<t\leq T.
\end{equation}
\end{proof}
And this will be enough to extend the strong solutions of $(H,\rho,u,\theta)$ beyond $t\geq \overline{T}$.

In truth, in view of the estimates obtained in  Lemmas \ref{s1c}-\ref{s7cc}, we quickly know that the functions $(H,\rho,u,\theta)|_{t=\overline{T}} =\lim_{t\rightarrow \overline{T}}(H,\rho,u,\theta)$ satisfy the conditions imposed on the initial data (\ref{th78})-(\ref{th79}) with $H_0\in W^{1,q}$. Therefore, we can take $(H,\rho,u,\theta)|_{t=\overline{T}}$ as the initial data and apply the local existence Theorem \ref{th5} to extend our local strong solution beyond $t\geq \overline{T}$. This contradicts the assumption on $\overline{T}$.

\section{Appendix}

We introduce some Poincar$\acute{\text{e}}$ type inequality  (see Chapter $8$ in \cite{lions}):
\begin{lemma}\label{pang}

There exists a constant $C$ depending only on $\Omega$  and $|\rho|_{r}$ $(r\geq 1)$ ($\rho\geq 0$  is a  real funciton satisfying $|\rho|_1>0$), 
such that for every $F\geq 0$ satisfying  
$$\rho F \in L^{1}(\Omega),\quad   \sqrt{\rho} F\in L^2(\Omega),\quad \nabla F\in  L^2(\Omega),$$ we have
\begin{equation*}
|F|_{6} \leq C\big( |\sqrt{\rho} F|_{2} +(1+|\rho|_2) |\nabla F|_{2} \big).
\end{equation*}
\end{lemma}

\begin{proof}
We first denote that
$$\overline{F}=\frac{1}{|\Omega|} \int_\Omega F(y) \text{d}y,$$
then via  the classical Poincar$\acute{\text{e}}$  inequality, we quickly deduce that
\begin{equation*}
\begin{split}
\overline{F} \int_\Omega \rho  \text{d}x&=\int_\Omega \rho     (F - \overline{F}) \text{d}x+\int_\Omega \rho     F \text{d}x \\
\leq& C \left( |\rho F|_{1} +|\rho|_2 |\nabla F|_{2} \right)\leq C \big(|\rho|^{\frac{1}{2}}_1 |\sqrt{\rho} F|_{2} +|\rho|_2 |\nabla F|_{2} \big),
\end{split}
\end{equation*}
which implies that 
\begin{equation}\label{zheng1}
\overline{F}\leq C \big( |\sqrt{\rho} F|_{2} +|\rho|_2 |\nabla F|_{2} \big).
\end{equation}
Second, we consider that 
\begin{equation}\label{zheng2}
\begin{split}
\|F\|_1=&|\nabla F|_2+|F|_2
\leq |\nabla F|_2+|F-\overline{F}|_2+\overline{F}|\Omega|^{\frac{1}{2}} \\
\leq & C\big( |\sqrt{\rho} F|_{2} +(1+|\rho|_2) |\nabla F|_{2} \big),
\end{split}
\end{equation}
then according to (\ref{zheng1})-(\ref{zheng2}) and the classical Sobolev imbedding theorem, we easily obtain the following inequality:
$$
|F|_6\leq C\|F\|_1\leq  C\big( |\sqrt{\rho} F|_{2} +(1+|\rho|_2) |\nabla F|_{2} \big).
$$
\end{proof}


\bigskip

\begin{thebibliography}{99}
\bibitem{TBK}T. Beale, T. Kato, A. Majda, Remarks on the breakdown of smooth solutions for the 3-D Euler equation, \textit{Comm. Math. Phys.}  \textbf{94} (1984) 61-66. \\[2pt]




\bibitem{CK3} Y. Cho, H. J. Choe, and H. Kim, Unique solvability of the initial boundary value problems for compressible viscous  fluids, \textit{J.Math.Pure.Appl.} \textbf{83} (2004) 243-275.\\[2pt]
\bibitem{CK} Y. Cho, H. Kim, Existence results for  viscous polytropic fluids with vacuum, \textit{J.Differ.Equations} \textbf{228} (2006) 377-411.\\[2pt]
\bibitem{guahu} Y. Cho, H. Kim, On classical solutions of the compressible Navier-Stokes equations with nonnegative initial densities, \textit{Manu.Math.} \textbf{120} (2006) 91-129.\\[2pt]

 \bibitem{y1} Y. Cho, B. Jin, Blow-up of viscous heat-conducting compressible flows, \textit{J.Math.Anal.Appl.} \textbf{320} (2006) 819-826.\\[2pt]



%
\bibitem{jishan} J. Fan, W. Yu, Strong solutions to the magnetohydrodynamic equations with vacuum, \textit{Nonl. Anal.} \textbf{10} (2009) 392-409.\\[2pt]

\bibitem{jishan1} J. Fan, S. Jiang, Y. Ou,  A blow-up criterion for compressible viscous heat-conductive flows.  \textit{Ann. I. H. Poincar-AN}  \textbf{27} (2010)  337-350.\\[2pt]


\bibitem{gandi} G. P. Galdi, \textit{An introduction to the Mathmatical Theorey of the Navier-Stokes equations}, Springer, New York, 1994.\\[2pt]




\bibitem{HX1} X.D. Huang, J. Li and  Z.P. Xin, Global Well-posedness of classical solutions with large oscillations and vacuum, \textit{ Comm. Pure. Appl. Math.} \textbf{65} (2012) 549-575.\\[2pt]

\bibitem{hup} X.D. Huang, J. Li and  Z.P. Xin, Blow-up criterion for the  compressible flows with vacuum states, textit{ Comm.  Math. Phys.} \textbf{301} (2010) 23-35.\\[2pt]
\bibitem{erwei} X.D. Huang, J. Li and  Z.P. Xin, Global Well-posedness of classical solutions to the Cauchy problem of two-dimensional barotropic compressible Navier-Stokes systems with  vacuum and large initial data, (2012) Preprint. \\[2pt]

\bibitem{huangxin} X. Huang, J. Li, Serrin-Type Blowup Criterion for Viscous, Compressible, and Heat Conducting Navier-Stokes and Magnetohydrodynamic Flows, \textit{ Comm.  Math. Phys.} \textbf{324} (2013) 147-171.\\[2pt]


\bibitem{mhd} H. Li, X. Xu and J. Zhang,  Global  classical solutions to the 3D compressible magnetohdrodynamic equations with large oscillations and   vacuum, \text{Siam. J. Math. Anal.} \text{45} (2013) 1356-1387. \\[2pt]

  \bibitem{lions} P. L. Lions, \textit{Mathematical topics in fluid dynamics} In: Compressible Models. Oxford University Press, \textbf{2} (1998) \\[2pt]

  \bibitem{mingtao}  M. Chen and S. Liu,  Blow-up criterion for 3D viscous-resistive compressible magnetohydrodynamic equations,  \textit{Manu.Math.} \textbf{36} (2012) 1145-1156.\\[2pt]


\bibitem{duyi}  Lu M., Du Y.  and Yao Z. A., Blowup criterion for compressible MHD equations, \textit{J. Math. Anal. Appl.} \textbf{379}  (2011) 425-438.\\[2pt]








\bibitem{pc} Ponce,  G. Remarks on a paper:  Remarks on the breakdown of smooth solutions for the $3$-D Euler equaitons, \textit{Comm. Math. Phys.} \textbf{98} (1985) 349-353.\\[2pt]





\bibitem{olga} O. Rozanova, Blow-up of smooth  solutions to  the barotropic compressible magnetohydrodynamic equations with finite mass and energy,  \textit{ Proceedings of Symposia in Applied Mathematics
} (2008). \\[2pt]




\bibitem{gerui} X. Xu and J. Zhang, A blow-up   criterion for the $3$-D non-resistive  compressible magnetohydrodynamic  Equations with initial vacuum, \textit{Nonl. Anal.}  \textbf{12} (2011) 3442-3451.\\[2pt]
\bibitem{zx} Z. P. Xin, Blow-up of smooth solutions to the compressible Navier-Stokes Equations with Compact Density, \textit{Commun.Pure.App.Math}\textbf{ 51} (1998) 0229-0240.\\[2pt]

\bibitem{xy} Z. P. Xin and W. Yan, On blow-up of classical solutions to the compressible Navier-Stokes Equations, \textit{Commun. Math. Phys.}  \textbf{321} (2013) 529-541. \\[2pt]







\end{thebibliography}
\end{document}